%% file: main.tex
\documentclass[a4paper]{article}

\usepackage{xcolor}
\usepackage[margin=1in]{geometry} 

\usepackage{blindtext}
\usepackage{parskip}

\usepackage{amsmath}
\usepackage{amsfonts}
\usepackage{amssymb}
\usepackage{amsthm}

\usepackage{dsfont}
\usepackage{stmaryrd}

\usepackage{comment}

\usepackage[utf8]{inputenc}
\usepackage{hyperref}
\hypersetup{
	unicode,
	pdfauthor={Benjamin Terlat},
	pdftitle={Regularity of the last passage percolation constant on directed acyclic graphs},
	pdfkeywords={Probability, Random graphs, last passage percolation, Interacting particle systems},
	pdfproducer={LaTeX},
	pdfcreator={pdflatex}
}

\usepackage[sort&compress,numbers,square]{natbib}

\usepackage{todonotes}

\usepackage{enumitem}

\theoremstyle{plain}
\newtheorem{theorem}{Theorem}[section]

\newtheorem{lemma}[theorem]{Lemma}

\newtheorem{proposition}[theorem]{Proposition}

\theoremstyle{definition}

\newtheorem{remark}[theorem]{Remark}

\numberwithin{theorem}{section}

\usepackage{graphicx, color}
\graphicspath{{fig/}}

\newcommand{\N}{\mathbb{N}}

\newcommand{\Z}{\mathbb{Z}}
\newcommand{\R}{\mathbb{R}}
\newcommand{\CC}{\mathbb{C}}

\newcommand{\E}{\mathbb{E}}
\newcommand{\Proba}{\mathbb{P}}
\newcommand{\SP}{\mathcal{S}\mathcal{P}}

\newcommand{\Ber}{\textrm{Bernoulli}}
\newcommand{\Bin}{\textrm{Binomial}}
\newcommand{\Tpast}{{T_{\text{past}}}}

\newcommand{\Mof}[1]{\mathcal{M}_1(#1)}

\newcommand\Set[2]{\{\,#1\mid#2\,\}}

\newcommand\lrSet[2]{\left\{\, #1 \mid #2 \, \right\}}

\newcommand{\fonction}[5]{ #1:\begin{array}{lrcl} 
    & #2 & \longrightarrow & #3 \\
& #4 & \longmapsto & #5 \end{array}}

\newcommand{\send}[3]{s^{\text{end}}_{#1}\left(#2 , #3\right)}

\usepackage{soul}

\usepackage{algorithm, algpseudocode} 
\usepackage{mathrsfs} 

\usepackage{tikz}
\usetikzlibrary{arrows, decorations.markings, plotmarks}
\tikzset{
 midarrow/.style={postaction={decorate,decoration={
        markings,
        mark=at position .5 with {\arrow[#1]{stealth}}
      }}},
}

\date{
	\today
}

\usepackage{authblk}

\newcommand*{\email}[1]{
    \normalsize\href{mailto:#1}{#1}\par
    }

\begin{document}

\title{Regularity of the time constant for last passage percolation on complete directed acyclic graphs}
\author{Benjamin Terlat \thanks{Partially supported by the Agence Nationale de la Recherche, Grant Number ANR 19-CE40-0025 (ANR ProGraM)}}
\affil{Université Paris-Saclay, CNRS, Laboratoire de mathématiques d'Orsay, 91405, Orsay, France. \\ Université Paris-Saclay, CNRS, CEA, Institut de Physique Théorique, 91191, Gif-sur-Yvette, France. \\ \email{benjamin.terlat@universite-paris-saclay.fr}}

	\maketitle
	\begin{abstract} 
    We study the time constant $C(\nu)$ of last passage percolation on the complete directed acyclic graph on the set of non-negative integers, where edges have i.i.d. weights with distribution $\nu$ with support included in $\{-\infty\}\cup\mathbb{R}$. 
    We show that $\nu \mapsto C(\nu)$ is strictly increasing in $\nu$. We also prove that $C(\nu)$ is continuous in $\nu$ for a large set of measures $\nu$. Furthermore, when $\nu$ is purely atomic, we show that $C(\nu)$ is analytic with respect to the weights of the atoms. In the special case of two positive atoms, it is an explicit rational function of these weights.
	\end{abstract}


\section{Introduction}
\label{sec:intro}

We study the problem of last passage percolation on a complete directed acyclic graph. Let $\Z^+$ denote the set of non-negative integers. We consider the directed graph with vertex set $\Z^+$ and (directed) edge set $\Set{(i,j)}{0 \leq i < j}$. Let $\nu$ be a probability distribution on $\{-\infty\}\cup\R$ and let $X = (X_{i,j})_{0\leq i < j}$ be i.i.d. random variables with distribution $\nu$. For $0 \leq i < j$, we assign weight $X_{i,j}$ to the directed edge $(i,j)$. 
We call \emph{directed path} a sequence of integers $ \pi  = (i_1, \dots, i_k)$ such that $0 \leq i_1 <  \dots < i_k$. The associated weight $w_{\pi}$ of the path $\pi$ is defined by 
$$w_{\pi} = X_{i_1, i_2} +  \dots + X_{ i_{k-1},i_k}.$$ 
The quantity we are interested in is the weight of a heaviest path starting at $0$ and ending at $n$ (see Figure \ref{fig:1}):
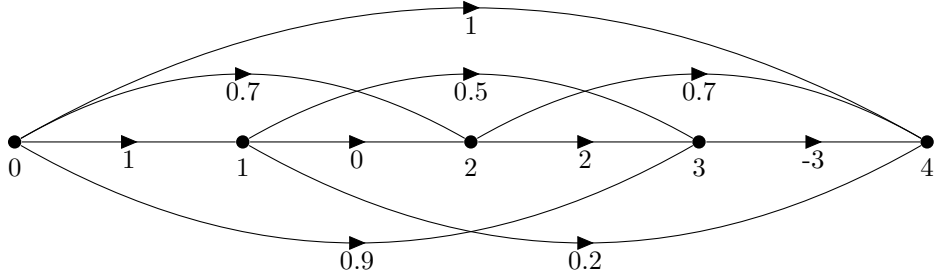
\begin{figure}[t]
\centering
\input{figures/fig1.tikz}
\caption{Illustration of the last passage percolation problem on the complete graph with $5$ vertices. With the weights given below each edge of this graph, $W_0 = 0$, $W_1=W_2=1$, $W_3 =3$, and $W_4=1.7$. A heaviest path from $0$ to $4$ is $\pi=(0,1,2,4)$. }
\label{fig:1}
\end{figure}
\begin{align}\label{eq:defWn}
W_n = \max \Set{ w_{\pi}}{ \pi = (i_1, \dots, i_k) \text{ is a directed path with } i_1=0 
\text{ and }i_k=n}.
\end{align} 
As usual with last passage percolation models, the sequence of last passage times $(W_n)_{n \geq 0}$ satisfies a subadditive property, namely:
$$W_{m} \geq W_n + W_{n,m}, $$ 
for all $n < m$ where 
\begin{align}\label{eq:defWn2}
W_{n,m} = \max \Set{ w_{\pi}}{ \pi = (i_1, \dots, i_k) \text{ directed path with } i_1=n \text{ and }i_k=m}
\end{align}
 has the same distribution as $W_{m-n}$.
Therefore, by Kingman's subadditive ergodic theorem, there exists a deterministic constant $C(\nu) \in \R\cup\{\pm\infty\}$, called the \emph{time constant for last passage percolation}, such that 
\begin{align}\label{eq:defC}
\frac{W_n}{n} \xrightarrow[n \xrightarrow{}\infty]{\text{a.s.}} C(\nu) .
\end{align}
The main goal of this article is to study the regularity of the function $\nu \mapsto C(\nu)$ and, in particular, study the analyticity of $C(\nu)$ with respect to the probabilities of some atoms of $\nu$. Indeed, in statistical physics, the analyticity of observables is of particular importance because a breach of analyticity usually pinpoints a phase transition in the model. A canonical example is the 
non-differentiability of the function $p\mapsto \theta(p)$, the probability that the origin belongs to an infinite cluster in classical Bernoulli bond percolation on $\Z^2$, at the critical value $p_c = \frac{1}{2}$. A similar result holds for the magnetization in the Ising model. One may refer to \cite{MR852458, MR874906} and \cite{MR894398} for more details. We also point out that results similar to those stated in this paper were previously obtained in \cite{MR4474502, MR4234130} for Bernoulli first-passage percolation on the lattice $\Z^d$ in the supercritical regime. 

Let us note that there is a trivial scaling property of the model when multiplying all the weights by a deterministic constant $M>0$: let $X$ be a random variable with distribution $\nu$ and let $\nu_M$  denote the distribution of the random variable $M X$. It holds that
\begin{align}\label{property:rescalingProperty}
C(\nu_M) =M \cdot C(\nu).
\end{align}
However, we point out that there is no such simplification when adding a deterministic constant to every weight (because the heaviest paths do not have a fixed number of edges). In particular, we do not know any simple relation between $C(\nu)$ and $C(\nu + M)$. 

For $x \in \{-\infty\}\cup\R$, let us denote by $\delta_x$ the Dirac measure at $x$. We remark that setting an edge with weight $-\infty$ is equivalent to removing this edge from the graph (at least for computing the heaviest paths).
As a consequence, when $\nu = (1-p)\delta_{-\infty} + p \delta_1$, the last passage percolation problem is equivalent to the study of the longest path for the Barak-{E}rd\H{o}s graph model \cite{MR763980}, which is a directed acyclic version of the {E}rd\H{o}s-Rényi model. 
This specific case has known applications in applied mathematics on stability in queues \cite{MR2028221}, in biology on food chains \cite{MR1164838,doi:10.1098/rspb.1986.0059}, and in computer science on parallel computing \cite{10.5555/324493.324628, MR1266246}. 
In \cite{MR4010962, MaRa2016}, Mallein and Ramassamy proved that the function $p \in (0,1] \mapsto C((1-p)\delta_{-\infty}+p\delta_1)$ is analytic, that its Taylor expansion around $1$ has integer coefficients, and they computed the first two terms of the asymptotic expansion of $p \mapsto C((1-p)\delta_{-\infty}+p\delta_1)$ around $0$. 
Those results are consequences of the coupling between last passage percolation on Barak-{E}rd\H{o}s graphs and the infinite-bin model, which is an interacting particle system introduced by Foss and Konstantopoulos in \cite{MR2028221}.

A simple generalization of the last passage percolation problem for Barak-{E}rd\H{o}s graphs is the case where  $\nu = (1-p) \delta_x + p \delta_1$ for $x \in \{\pm \infty\}\cup\R$. It is proved in \cite{https://doi.org/10.48550/arxiv.2006.01727} that the time constant $C((1-p)\delta_x + p\delta_1)$ seen as a function of $x$ is strictly increasing and convex for all $p \in (0,1)$, and that its non-differentiability points are the non-negative integers, the integers greater than one and their reciprocals.

Kingman's subadditive theorem ensures only that $C(\nu) \in \R \cup \{ \pm\infty\}$ so that the time constant may possibly be infinite. Let 
$$
M_{\nu} = \inf\{t \in \R, \nu([t,+\infty))=0\}
$$
denote the essential supremum of $\nu$. Obviously, we have $W_n \leq n M_{\nu}$ almost surely and therefore $C(\nu)$ is necessarily finite wherever $\nu$ has upper-bounded support. A less restrictive sufficient condition to ensure that $C(\nu) \in [0,+\infty)$ is given in \cite{MR3161646} for probability distributions with support included in $\{- \infty\}\cup \R_+$. It states that $C(\nu)$ is finite when $\E\left[ \max(0,X_{1,2})^2 \right] < +\infty$. Reciprocally, it has been shown that if $\nu$ is regularly varying with index $s \in (0,2)$ (which implies in particular that $\E\left[ \max(0,X_{1,2})^2 \right] = +\infty$), then $C(\nu)=+\infty$. In fact, this result can easily be extended to all probability distributions $\nu$ on $\{-\infty\}\cup\R$ by approximations and using the monotonicity of $\nu \mapsto C(\nu)$. However, the study of the time constant $C(\nu)$ is greatly simplified when the support of the weight distribution is upper-bounded and this will be our standing assumption throughout the paper. 

Let us mention that the assumption that $\nu$ has finite essential supremum was previously enforced in a work by Foss, Konstantopoulos, Mallein and Ramassamy in \cite{https://doi.org/10.48550/arxiv.2110.01559} where they used this fact to study the last passage percolation model via a coupling with an interacting particle system called \emph{max growth system} (MGS) which is a generalization of the infinite-bin model introduced in \cite{MR2028221}.  Using tools from renovation theory, they proved that it is possible to perform perfect simulations of the time constant $C(\nu)$ when $M_{\nu}< +\infty$. This connection between last passage percolation on a directed complete graph and the MGS is also an essential ingredient for several proofs of this paper.

The remainder of the introduction is devoted to presenting the main results proved in the later sections of this paper.

We first study the monotonicity of the function $\nu \mapsto C(\nu)$. If $\nu_1$ and $\nu_2$ are two probability distributions on $\{-\infty\}\cup\R$ such that $\nu_1$ is stochastically dominated by $\nu_2$, then $C(\nu_1) \leq C(\nu_2)$ by a trivial coupling. This means that the time constant $C(\nu)$ is a non-decreasing function for the stochastic (partial) order. The more delicate question is whether it is \emph{strictly} increasing. Our first result provides a positive answer and generalizes a previous monotonicity result stated in \cite{https://doi.org/10.48550/arxiv.2006.01727} concerning the function $m \mapsto C((1-p)\delta_m + p\delta_M)$. 

Let us denote by $\mathcal{M}_1$ the set of all probability distributions on $(\{-\infty\}\cup\R, \mathcal{B}(\{-\infty\}\cup\R))$, where $\mathcal{B}(\{-\infty\}\cup\R)$ is the Borel algebra for the usual topology on $\{-\infty\}\cup\R$. For every Borel set $A \in \mathcal{B}(\{-\infty\} \cup \R)$, let us denote by $\Mof{A}$ the set of probability distributions $\nu \in \mathcal{M}_1$ such that $\nu(A) = 1$. 
\begin{theorem}\label{theorem:strictIncreasing}
For all $M>0$, $ C(\nu)$ is strictly increasing for the stochastic order on the set of probability distributions $\nu \in \Mof{[-\infty, M]}$  
such that $\nu(\{M\})>0$ .
\end{theorem}

\begin{remark}
\begin{itemize}
\item The theorem above tells us something about the geometry of the heaviest paths: because the time constant is strictly monotonic, edge weights with arbitrarily large negative values in the support of $\nu$ may contribute to a heaviest path. Thus, avoiding all the edges with weights below a certain cutoff is not the optimal strategy, as one could have naively expected.
\item Let also remark that the assumption that  $\nu(\{M\})>0$ with $M>0$ cannot be dropped without further assumptions because the strict monotonicity does not hold anymore for probability distributions  $\nu \in \Mof{[-\infty, 0]}$.  
Indeed, for any such distribution different from $\delta_{-\infty}$, it is easy to check that $C(\nu) = 0$ (because in that case, every step is a penalty so the best strategy is to go from $0$ to $n$ with only a finite number of steps that does not increase to infinity with $n$), c.f. \cite{https://doi.org/10.48550/arxiv.2110.01559} for details. 
\end{itemize}
\end{remark}

Denote by $\mathcal{M}_1^b = \Set{\nu \in \mathcal{M}_1}{M_{\nu} < \infty}$ the set of all probability distributions in $ \mathcal{M}_1$ with finite essential supremum and by $d_{LP}$ the L\'evy–Prokhorov metric on $\mathcal{M}_1$. It is known that this metric corresponds to the topology of weak convergence of measures when working on a separable space, which is the case here. For $\nu_1,\nu_2 \in \mathcal{M}_1$, we set 
\begin{align}\label{eq:myMetric}
    d(\nu_1,\nu_2)= \max(d_{LP}(\nu_1,\nu_2), |M_{\nu_1}-M_{\nu_2}|).
\end{align}
It is straightforward that $d$ defines a metric on $\mathcal{M}_1$. 
\begin{theorem}\label{theorem:continuity}
The map $ \nu \in \mathcal{M}_1^b \mapsto C(\nu) $ is continuous for the metric $d$ defined in \eqref{eq:myMetric}.
\end{theorem}

\begin{remark}
One may notice that $\nu \mapsto C(\nu)$ is not continuous on  $\mathcal{M}_1$ for the L\'evy-Prokhorov metric. With $\nu_n = (1-\varepsilon_n)\delta_1 +\varepsilon_n\delta_n $ and $\nu=\delta_1$, $\nu_n$ converges to $\nu$ for the L\'evy-Prokhorov metric when $\varepsilon_n$ tends to $0$ as $n$ tends to infinity. 
However, with $C_0(p):= C((1-p)\delta_0 + p\delta_1)$, $C_0^{-1}$ its inverse and $\varepsilon_n=(C_0^{-1}(2n^{-1}))$,
one can show that $\liminf_{n \rightarrow \infty}C(\nu_n) \geq 2$ using the rescaling property \eqref{property:rescalingProperty}, the monotonicity of the time constant and the fact that $C_0$ is continuous and takes value $0$ at $p=0$. Notice that the inverse of $C_0$ exists since $C_0$ is increasing by Theorem \ref{theorem:strictIncreasing}. 
\end{remark}

\begin{remark}
The fact that the essential supremum is an atom for the measures 
considered is crucial for our construction. 
This is why we consider the metric $d$ instead of the L\'evy-Prokhorov metric. It might be possible to prove the continuity for the L\'evy-Prokhorov metric of the map $\nu \mapsto C(\nu)$ on $\Set{\nu \in \mathcal{M}_1}{ C(\nu) \leq M}$ for any fixed $M>0$, but this might require  a different construction. 
\end{remark}

The following technical result is a key ingredient for proving Theorem \ref{theorem:continuity}. 

\begin{theorem}\label{theorem:analyticCpmu}
For all $M \in\R^*_+$ and $\mu \in \Mof{[-\infty,M)}$, 
the map $p \mapsto C((1-p)\mu + p\delta_M)$ is analytic on $(0,1]$.
\end{theorem}

This result is also interesting in itself since it shows that it would suffice to know the coefficients of the Taylor expansion in $p$ of the time constant at some $p_0 \in (0,1]$ to obtain $C((1-p)\mu + p\delta_M)$ for all $p\in(0,1]$. The scope of Theorem \ref{theorem:analyticCpmu} is rather general since every distribution with finite essential support can be decomposed as $(1-p)\mu+p\delta_M$ for some $M\in \{-\infty\}\cup \R$,  $p\in [0,1]$, and $\mu\in\Mof{[-\infty,M)}.$ 

In addition, we give a lower bound for the radius of convergence of $ p \mapsto C((1-p)\mu + p\delta_M)$ at $p=1$ in Remark \ref{remark:radius}, which improves the lower bound given in \cite{MaRa2016} in the case where $\mu = \delta_{-\infty}$.

In \cite{MR4010962}, Theorem \ref{theorem:analyticCpmu} was proved in the specific case where $\mu = \delta_{-\infty}$, which corresponds to the Barak-{E}rd\H{o}s case. Our proof of Theorem \ref{theorem:analyticCpmu} is an extension of the analyticity proof of \cite{MR4010962}. It relies on 
a formula for $C((1-p)\mu + p \delta_M)$ as a sum, over an infinite class of words, of polynomials 
in $p$ whose coefficients depend on $\mu$ (Proposition \ref{proposition:formuleMGS}).

For measures with finite support, we obtain the following result, which is a generalization in a different direction of the aforementioned analyticity result of \cite{MR4010962}: 
\begin{theorem}\label{theorem:analyticityFiniteSupport}
For all $N\in \N$ and $a_1  > a_2 > \dots > a_N \geq - \infty$, the map $$(p_2, \dots, p_N) \mapsto C\left(\sum_{i=1}^N p_i \delta_{a_i} \right)$$ is analytic on $ \Set{(p_2, \dots, p_N) \in [0,1]^N}{0 \leq p_2 + \dots + p_N < 1}$, where $p_1 = 1-(p_2+\dots + p_N)$.
\end{theorem}

\begin{remark}
In general, we cannot extend the analyticity of the function considered in Theorem \ref{theorem:analyticityFiniteSupport} to the set $ \Set{(p_2, \dots, p_N) \in [0,1]^N}{0 \leq p_2 + \dots + p_N \leq 1}$. For instance, $p \mapsto C((1-p)\delta_{-\infty}+p\delta_1)$ is not analytic at $0$ cf. \cite{MR4010962, MaRa2016}.
\end{remark}

Until now, we have considered general measures, possibly taking negative values. 
The model is easier to study when the support of $\nu$ is included in $\R^*_+$. For measures $\nu$ of the form $(1-p)\delta_m + p\delta_M$ with $0 < m \leq M $, the study of the time constant can be reduced to the study of a Markov chain on a finite state space, and we have the following result: 
\begin{theorem}\label{theorem:rationality2atoms}
For all $0<m \leq M$, $p \mapsto C((1-p)\delta_m + p\delta_M)$ is a rational function on $[0,1]$.
\end{theorem}

In addition, we can compute the explicit numerical expression for $C((1-p)\delta_m + p\delta_M)$ for all positive values of $m$ and $M$.  The complexity of the computation increases with $\frac{M}{m}$.
We provide in Subsections \ref{subsec:2atom-caseinvk} and \ref{subsec:2atom-FKP} these numerical expressions for $m$ in $[\frac{M}{5}, M]$.

When $m=0$, one can write the time constant as the reciprocal of the Ramanujan $\Psi$-function  \cite{dutta2020limit}. We give in this article an alternative proof of that result:
\begin{theorem}[\cite{dutta2020limit}]\label{theorem:fromDutta}
For all $M > 0$ and $p\in]0,1],$  $$C((1-p)\delta_{0} + p \delta_M)=M \left( \sum_{n=1}^{\infty} (1-p)^{\binom{n}{2}}\right)^{-1} .$$
\end{theorem}

\subsection*{Organization of the paper}

In Section \ref{sec:stationarity}, we recall the coupling from \cite{https://doi.org/10.48550/arxiv.2110.01559} between last passage percolation and the MGS. We construct a time-stationary version of the MGS for measures $(1-p)\mu +p\delta_M $ with $p \in (0,1]$, $M\in \R_+^*$ and $\mu \in \Mof{[- \infty , M)}$ 
(Proposition \ref{proposition:stationaryMGS}). To do so, we use different (yet of similar flavor) renovation events from those in \cite{https://doi.org/10.48550/arxiv.2110.01559} and \cite{MR4010962}. With these, we obtain a formula for $C(\nu)$ in Proposition \ref{proposition:formuleMGS} that we use in Section \ref{sec:regularity} to prove the regularity results on $\nu \mapsto C(\nu)$. 

In Section \ref{sec:regularity}, we study the regularity of $\nu \mapsto C(\nu)$. The main results proved in this section are Theorem \ref{theorem:analyticCpmu}, Theorem \ref{theorem:analyticityFiniteSupport}, Theorem \ref{theorem:continuity} and Theorem \ref{theorem:strictIncreasing}.

In Section \ref{sec:atomicMeasures}, we study the case of measures supported by two elements. Firstly, we prove Theorem \ref{theorem:rationality2atoms} using results from \cite{https://doi.org/10.48550/arxiv.2006.01727} and we give the expression of $C((1-p)\delta_m + p \delta_M)$ for values of $m$ close to $M$. In the last subsection, we give an alternative proof of Theorem \ref{theorem:fromDutta} using the stationary MGS.

\section{Stationary max growth system and a formula for \texorpdfstring{$C(\nu)$}{TEXT}}
\label{sec:stationarity}

This section is dedicated to the max growth system (MGS) which was introduced in \cite{https://doi.org/10.48550/arxiv.2110.01559}. In Subsection \ref{subsec:MGS}, we recall the max growth system.  This model was introduced by Foss, Konstantopoulos, Mallein and Ramassamy in order to perform perfect simulation of the time constant $C(\nu)$. In Subsection \ref{subsec:statMGS}, we construct a stationary version of the MGS using tools from renovation theory. This construction differs from the one in \cite{https://doi.org/10.48550/arxiv.2110.01559}. In Subsection \ref{subsec:formulaC}, we use this construction in order to obtain a new formula for $C(\nu)$ as an infinite sum of quantities depending on $\nu$.

\subsection{The max growth system}
\label{subsec:MGS}

The max growth system (MGS) is an interacting particle system in discrete time on $\{-\infty\}\cup\R$  introduced in \cite{https://doi.org/10.48550/arxiv.2110.01559}. A new particle is added to the configuration at each step according to a specific Markovian dynamics.

We consider $\mathcal{N}$ the set of atomic measures $\lambda$ on $  \{-\infty\}\cup \R $ taking values in $\Z_+$ and such that 
\[
\forall t \in \R, \lambda([t,+ \infty)) < +\infty.
\]
An element of $\mathcal{N}$ represents a configuration of particles on the real line in the following sense: if we consider $\lambda \in \mathcal{N}$ and $A \in \mathcal{B}(\R)$, $\lambda(A)$ represents the number of particles with positions in $A$ in the configuration $\lambda$. For instance, $\delta_{-1} + 2 \delta_0$ is an element of $\mathcal{N}$. 
It corresponds to the configuration with one particle at position $-1$ and two at position $0$. 

For $\lambda \in \mathcal{N}$, we set $\lambda_1 \geq \lambda_2 \geq \dots$ the positions of the particles in $\lambda$ in decreasing order (with repetitions). For example, $\lambda = \delta_{-1} + 2 \delta_0$, we have $\lambda_1 = \lambda_2 = 0 $ and $\lambda_3 = -1$. Notice that $ \lambda(\R) $ is the number of particles in the configuration $\lambda$.

We set $\mathcal{W} = (\{- \infty \}\cup\R)^{\N}$ to be the set of weights. For $x = (x_i)_{i \geq 1} \in \mathcal{W}  $ and $\lambda \in \mathcal{N}$, we set \begin{equation*}
    \mathfrak{m}(\lambda,x) = \max_{1 \leq i \leq  \lambda(\R)}(\lambda_i + x_i).
\end{equation*}
The deterministic dynamics of the MGS is defined as follows: the weights sequence $x$ applied to the configuration $\lambda$ adds a particle to $\lambda$ at position $\mathfrak{m}(\lambda,x)$. Therefore, for all $x \in \mathcal{W}$, we define the map $\Phi_x : \mathcal{N} \rightarrow \mathcal{N}$ such that for all $\lambda \in \mathcal{N}$, \begin{equation*}
\Phi_x(\lambda) = \lambda + \delta_{\mathfrak{m}(\lambda,x)}.
\end{equation*}
See Figure \ref{fig:2} for an example.

\begin{figure}[t]
\centering
\begin{minipage}{.45\textwidth}
    \centering
    \input{figures/fig2-1.tikz}
\end{minipage}
\begin{minipage}{.45\textwidth}
    \centering
    \input{figures/fig2-2.tikz}
\end{minipage}
\caption{On the left, the configuration $\lambda = \delta_{-1} + 2\delta_{0.3} + \delta_1$. On the right, the configuration $\Phi_x(\lambda)$, with $x=(0.5,-0.3,1.7,1, \dots)$. Here, $\mathfrak{m}(\lambda, x) = 0.3 + 1.7 = 2$.}
\label{fig:2}
\end{figure}
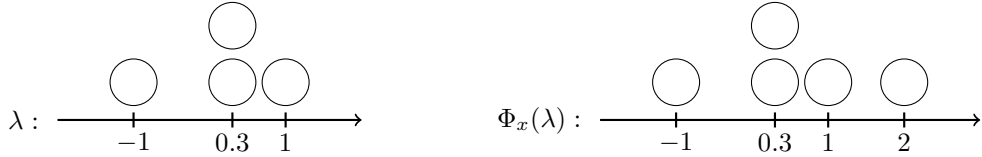

An \emph{MGS with starting configuration $\lambda^{(0)} \in \mathcal{N}$ and weight distribution $\nu$} is a process $(\lambda^{(n)})_{n \geq 0}$ taking values in $\mathcal{N}$ such that $ \lambda^{(n+1)} = \Phi_{X^{(n+1)}} (\lambda^{(n)}) $ for all $n\in\Z_+$, where $X^{(n)} = (X^{(n)}_{i})_{i\geq 1}$ and $(X_i^{(n)})_{(i,n)\in\N^2}$ are i.i.d. random variables with distribution $\nu$.

We are interested in this process because of the following coupling from \cite{https://doi.org/10.48550/arxiv.2110.01559} relating the MGS to last passage percolation on complete directed acyclic graphs. It is an extension of the infinite-bin model introduced in \cite{MR2028221}.  
\begin{proposition}[\cite{https://doi.org/10.48550/arxiv.2110.01559}]\label{proposition:couplingProperty}

We consider $W_n$ as defined in \eqref{eq:defWn} where $X_{i,j}$ has distribution $\nu$ for all $0 \leq i<j$. Then, if we set $$ \lambda^{(n)} = \sum_{j=0}^n \delta_{ W_j}$$ for all $n \geq 0$, then 
$(\lambda^{(n)})_{n \geq 0}$ is an MGS with starting configuration $\lambda^{(0)}=\delta_0$ and weight distribution $\nu$. 
In particular, $\lambda_1^{(n)}=\max\{W_j, j\in\llbracket0,n\rrbracket\}$.
\end{proposition}

See Figure \ref{fig:3} for an illustration of this coupling. Since $\nu$ has finite essential supremum, $\E[(X_{1,2})_+]$ is finite. 
We give in Remark \ref{remark:notDep} a self-contained proof of the fact that $(n^{-1}\max\{W_j, j\in\llbracket 0,n\rrbracket\})$ and $(n^{-1}W_n)$ have the same limit almost surely and in $L^1$. 
Then, by Proposition \ref{proposition:couplingProperty}:
\begin{align}\label{eq:convCouplingFK}
    \frac{\lambda_1^{(n)}}{n} \xrightarrow[n \rightarrow \infty]{\text{a.s, } L^1} C(\nu).
\end{align}
This result is also derived in  \cite{https://doi.org/10.48550/arxiv.2110.01559}. 

\begin{figure}[t]
\centering
\begin{minipage}{.45\textwidth}
    \centering
    \input{figures/fig3-1.tikz}
\end{minipage}
\begin{minipage}{.45\textwidth}
    \centering
    \input{figures/fig3-2.tikz}
\end{minipage}
\caption{On the left, the weighted graph from Figure \ref{fig:1}. On the right, the corresponding configuration for the MGS via the coupling from Proposition \ref{proposition:couplingProperty}. 
}
\label{fig:3}
\end{figure}
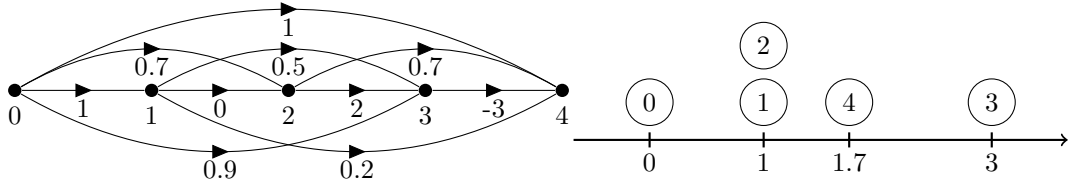

 \begin{remark}
 In fact, the limit \eqref{eq:convCouplingFK} also holds for any MGS with weight distribution $\nu$, not necessarily with starting configuration $\delta_0$.  We explain this fact more in details in Remark \ref{remark:indepInit}.
 \end{remark}
 
\subsection{ Stationary MGS for weight distributions with an atom at their essential supremum}
\label{subsec:statMGS}

For the remainder of this section, we assume that $\nu$ has the form $(1-p)\mu + p \delta_M$ with $M>0$, $ p \in (0,1] $ and $\mu$ a probability distribution on $[-\infty,M)$. We construct a stationary (in time) version of the max growth system (MGS) with weight distribution $\nu$. 
Similar constructions have been done in \cite{MR2028221, MR4010962} for the infinite-bin model using tools from extended renovation theory. 
Another similar construction have been done  for the observable describing the speed of the front in the MGS in  \cite{https://doi.org/10.48550/arxiv.2110.01559}. 
In both cases, the construction of the stationary process comes from renovation events, as introduced in \cite{MR2028221}, which consists in random times at which the future becomes independent from the past. 

Let us denote by $F_s(\lambda)$ the configuration obtained by shifting all the particles in $\lambda$ by $s\in\R$. In other words, $(F_s(\lambda))_i = \lambda_i - s$ for all $i$. 
\begin{proposition}[Time-stationary MGS]
\label{proposition:stationaryMGS}
We consider $M\in\R^*_+$ and $(X_i^{(n)})_{i \in \N, n \in \Z}$ i.i.d. random variables with distribution $\nu$ on $[-\infty,M]$ such that $\nu(\{M\}):=p>0$. Almost surely, there exists a unique process $( \widetilde{\lambda}^{(n)} )_{n \in \Z}$ taking values in $\mathcal{N}$ such that:
\begin{enumerate}[label=\emph{A.\arabic*}]
    \item $\forall n \in \Z, \,\widetilde{\lambda}^{(n+1)} = \Phi_{X^{(n+1)}}(\widetilde{\lambda}^{(n)})$, \label{enum:1}
    \item $\widetilde{\lambda}_1^{(0)}=0$.\label{enum:2}
\end{enumerate}
In addition, $F_{\widetilde{\lambda}^{(n)}_1}(\widetilde{\lambda}^{(n)})$ is $\mathcal{F}_n$-measurable for all $n \in \Z$, where $\mathcal{F}_n$ is the $\sigma$-algebra generated by $(X_i^{(k)})_{i \in \N, k \leq n}$. 
\end{proposition}

\begin{remark}\label{remark:stat}
The process $(\widetilde{\lambda}^{(n)})_{n \in \Z}$ is stationary in the sense that it is invariant in time. Since $(X_i^{(n)})_{i \in \N, n \in \Z}$ has the same distribution as $(X_i^{(n+1)})_{i \in \N, n \in \Z}$, the processes obtained via Proposition \ref{proposition:stationaryMGS} for those two weights sequences have the same distribution. 
\end{remark}

\begin{remark}
In \cite{https://doi.org/10.48550/arxiv.2110.01559}, a similar construction of the stationary MGS is done for all measures with $p\in[0,1]$, which is a more general case. However, we do not obtain the formula from Proposition \ref{proposition:formuleMGS} from that other construction.  
In \cite{MR4010962}, the weight distribution has form $p \delta_1 + (1-p) \delta_{-\infty}$, and the renovation events consist in “good” and “bad” words appearing in a sequence of i.i.d. random variables with geometric distribution. The construction leads to a formula for $C((1-p)\delta_{-\infty}+p\delta_1 )$ as a sum, over an infinite class of words, of polynomials in $p$. This formula leads to the proof of the analyticity of $p \mapsto C((1-p)\delta_{-\infty} + p\delta_1)$. 
The renovation event used in this article are similar in shape to those of \cite{https://doi.org/10.48550/arxiv.2110.01559}. However, the construction done in this article using those events looks more like what have been done in \cite{MR4010962}. 
In our case, the renovation events we use  lead us to a formula for $C(\nu)$ as a sum, over an infinite class of words, of polynomials in $p$ whose coefficients depend on $\mu$. 
\end{remark}

The proof of Proposition \ref{proposition:stationaryMGS} relies on the existence of renovation events. They consist in particular sequences of weights, which are introduced in the following Lemma \ref{lemma:keyLemmaTr}. Firstly, let us introduce some necessary notations.

A \emph{word} is defined as a finite sequence of positive integers. Let  $\mathcal{A} = \cup_{n \in \Z_+} \N^n$ denote the set of all words. 
For $\alpha  =(\alpha_1, \dots, \alpha_n) \in \mathcal{A}$, $|\alpha|:=n$ is the \emph{length} of the word $\alpha$. 
We denote by $\mathcal{T} = \{ \alpha \in \mathcal{A} \mid \forall i \in \llbracket 1 , |\alpha | \rrbracket, \alpha_i \leq i\} $ the set of all \emph{triangular words}.

For $l\in\N$ and $x=(x^{(1)},\dots,x^{(l)}) \in \mathcal{W}^l$, we set $\Phi_{x}:= {\Phi_{x^{(l)}}} \circ \dots \circ  \Phi_{x^{(1)}} $ the map corresponding to the successive applications of the weight sequences $x^{(1)}, x^{(2)}, \dots, x^{(l)}$ on configurations in $\mathcal{N}$. 
Let us define $\alpha(x)$ \emph{the word associated to the weight sequence} $x$ as the word of length $l$, for which the $n$-th letter corresponds to the index of the first occurrence of the value $M$ in the sequence $ x^{(n)}$. More explicitly, for all $n\in\llbracket 1 , l\rrbracket$, $$ \alpha_n(x) := \inf\{i \in \N, x^{(n)}_i = M\}. $$

Since $\nu(\{M\})>0$, it follows that $\alpha_i( X^{(a)}, \dots , X^{(b)} )$ is finite a.s. for all $a < b \in \Z$ and $i \in \llbracket a, b \rrbracket$. We denote by $\mathcal{W}_{\leq M} := [- \infty, M]^{\N}$ the set of all the elements $x \in \mathcal{W}$ such that $x_i \leq M$ for all $i\geq 1$.

\begin{lemma}\label{lemma:keyLemmaTr}
For all $l\in\N$ and $x = (x^{(1)}, \dots, x^{(l)}  ) \in (\mathcal{W}_{\leq M})^l$ such that $\alpha(x) \in \mathcal{T}$, the quantity $\mathfrak{m}(\Phi_{(x^{(1)},\dots, x^{(l-1)})}(\lambda), x^{(l)})-\lambda_1$ depends only on $\alpha(x)$ and $(x^{(i)}_j)_{i \in \llbracket 1 , l \rrbracket, j \in \llbracket 1 , \alpha_i(x) - 1\rrbracket}$. In particular, it does not depend on the configuration $\lambda$. In addition, $\mathfrak{m}(\Phi_{(x^{(1)},\dots, x^{(i-1)})}(\lambda), x^{(i)})-\lambda_1 \geq M$ for all $i \in \llbracket 1 , l \rrbracket$.
\end{lemma}
We illustrate this Lemma \ref{lemma:keyLemmaTr} in Figure \ref{fig:4}. 

\begin{figure}[!t]
\centering
\begin{minipage}{.45\textwidth}
    \centering
    \input{figures/fig4-1.tikz}
\end{minipage}
\begin{minipage}{.45\textwidth}
    \centering
    \input{figures/fig4-2.tikz}
\end{minipage}
\caption{Illustration of Lemma \ref{lemma:keyLemmaTr} for $M=1$. The configurations $\Phi_x(\lambda)$ and $\Phi_x(\lambda')$ where $\lambda=\delta_0$, $\lambda'= \delta_2 + 2 \delta_3$ and $x=(x^{(1)}, x^{(2)}, x^{(3)},x^{(4)}) $, with $x^{(1)} = (1, 0.2, -3, \dots)$,  $x^{(2)} = (0.5, 1, 0.5,\dots)$, $x^{(3)} = (-0.2, 1, 0.3, 1 ,\dots)$ and $x^{(4)} = (-1.5, 0.5, 0, 1, \dots)$. The configurations $\lambda$ and $\lambda'$ are represented in black. The particles added are in blue and indexed by their order of appearance (the particle indexed by $i$ corresponds to the particle obtained from the weight sequence $x^{(i)}$). Here, changing the values appearing after the first $1$ in the sequences $x^{(i)}$ would not change the obtained configurations. In this case, the word $\alpha(x)=(1,2,2,4)$ is indeed triangular.}
\label{fig:4}
\end{figure}
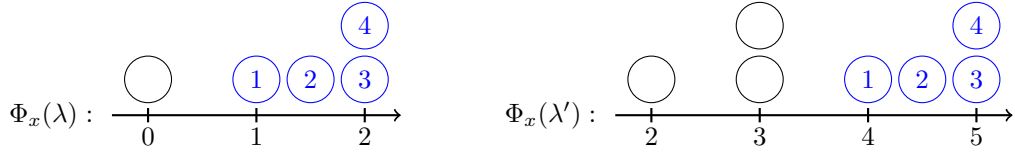

\begin{proof}
 We proceed by induction on $l$. We consider $x = x^{(1)} \in \mathcal{W}_{\leq M}$ such that $\alpha(x) \in \mathcal{T}$.  We have $x^{(1)}_1 = M$ because $\alpha(x)$ is a triangular word. 
In this case, $\mathfrak{m}(\lambda, x^{(1)})-\lambda_1=M$ and the result is proved for $l=1$.

We now assume that Lemma \ref{lemma:keyLemmaTr} holds for some $l \in \N$. We consider $x = (x^{(1)}, \dots, x^{(l+1)}  ) \in (\mathcal{W}_{\leq M})^{l+1}$ such that $\alpha(x) \in \mathcal{T}$. 
Since $\alpha(x) \in \mathcal{T}$, there exists $k_0 \in \llbracket 1 , l+1 \rrbracket$ such that $x^{(l+1)}_{k_0} = M $. In addition, $x^{(l+1)} \in \mathcal{W}_{\leq M}$. Therefore, $ x^{(l+1)}_i \leq x^{(l+1)}_{k_0} $ for all $i \geq k_0$. 
As a consequence, 
$$\mathfrak{m}(\Phi_{x'}(\lambda), x^{(l+1)}) = \max_{1 \leq i \leq k_0}((\Phi_{x'}(\lambda))_i + x^{(l+1)}_i )$$
where $x' = (x^{(1)}, \dots, x^{(l)})$. We set $y_i = \mathfrak{m}(\Phi_{(x^{(1)},\dots, x^{(i-1)})}(\lambda), x^{(i)}) - \lambda_1 $ for all $i \in \llbracket 1 , l \rrbracket$, $y_{l+1} = 0$ and $\varphi$ a permutation of $\llbracket 1, k_0 \rrbracket$ such that $ y_{\varphi(1)} \geq \dots \geq y_{\varphi(k_0)}$. 
By induction hypothesis applied to $x'$, for all $\lambda \in \mathcal{N}$ and  $i  \in \llbracket 1, l \rrbracket$, $ y_i$ depends only on $(\alpha_1(x), \dots, \alpha_l(x))$ and  $(x^{(i)}_j)_{i \in \llbracket 1 , l \rrbracket, j \in \llbracket 1 , \alpha_i(x)-1 \rrbracket}$. 
In addition, $y_i \geq M > 0$ for all $i\in \llbracket 1 , l \rrbracket$. Then, we obtain 
\begin{equation*}
    \mathfrak{m}(\Phi_{x'}(\lambda), x^{(l+1)}) = \max_{1 \leq i \leq k_0}(\lambda_1 + y_{\varphi(i)} + x^{(l+1)}_i ).
\end{equation*} Since $y_1, \dots, y_{l}$ depends only on $(\alpha_1(x), \dots, \alpha_l(x))$ and $(x_j^{(i)})_{i \in \llbracket 1 , l \rrbracket, j \in \llbracket 1 , \alpha_i(x) - 1\rrbracket}$,  $\mathfrak{m}(\Phi_{x'}(\lambda), x^{(l+1)}) - \lambda_1$ depends only on $\alpha(x)$ and $(x^{(i)}_j)_{i \in \llbracket 1 , l+1 \rrbracket, j \in \llbracket 1 , \alpha_i(x) - 1\rrbracket}$. 
In addition, $y_i \geq 0$ for all $i$. Then, $\mathfrak{m}(\Phi_{x'}(\lambda), x^{(l+1)})-\lambda_1 \geq y_{\varphi(k_0)} + x^{(l+1)}_{k_0} \geq x^{(l+1)}_{k_0} = M $, which concludes the inductive step.
\end{proof}

\begin{remark}\label{remark:keyLemmaTr}
As a consequence of Lemma \ref{lemma:keyLemmaTr}, for all $\lambda \in \mathcal{N}$ and $x \in (\mathcal{W}_{\leq M})^l$ such that $\alpha(x) \in \mathcal{T}$, we have 
$$ \Phi_x(\lambda) = \lambda + \sum_{i=1}^l \delta_{\lambda_1 + y_i},$$ 
where 
$y_i := \mathfrak{m}(\Phi_{(x^{(1)},\dots, x^{(i-1)})}(\lambda), x^{(i)})-\lambda_1 \geq M>0$ for all $i \in \llbracket 1 , l \rrbracket$. 
As a consequence, for $(x_j^{(i)})_{(i,j)\in\N^2}$ such that $\alpha ((x^{(1)}, \dots , x^{(l)}) ) \in \mathcal{T}$ for all $l \in \N$,  if we know only $\lambda_1$ and $(x^{(i)}_j)_{i \in \llbracket 1, l \rrbracket, j \in \N}$, then we can obtain the positions of all the particles with positions greater than $\lambda_1$ in $\Phi_{(x^{(1)}, \dots, x^{(l)})}(\lambda)$ for all $l \in \N$. 
\end{remark}

We use Lemma \ref{lemma:keyLemmaTr} to construct the stationary version of the MGS.
To do so, we set $\xi_n = \inf\{i \in \N, X^{(n)}_i = M\}$ for all $n \in \Z$, where $(X^{(n)}_i)_{\substack{i \in \N, n\in\Z}}$ are i.i.d. random variables with distribution $\nu$. 
We now prove that there are almost surely infinite triangular words in the sequence $(\xi_n)_{n \in \Z}$. 
We consider $\mathcal{R} =  \Set {n \in \Z}{ \forall i \in \N,\xi_{n + i-1} \leq i}$ the set of all the times at which an infinite triangular word starts. By the previous Remark \ref{remark:keyLemmaTr}, those times are renovation events in the sense that, for all $n \in \mathcal{R}$, we are able to reconstruct the future of the process after time $n$ knowing only the front position at this time.  

\begin{lemma}\label{lemma:TpastFinite}
For $p \in (0,1]$, $\inf \mathcal{R} = -\infty$ and $\sup \mathcal{R} = +\infty$ almost surely.
\end{lemma}
\begin{proof}
We notice that the probability distribution of the system is invariant by the shift map 
$$\fonction{\sigma}{\R^{\Z \times \N}}{{\R^{\Z \times \N}}}{ (x_i^{(n)})_{n \in \Z, i \in \N}}{ (x_i^{(n+1)})_{n \in \Z, i \in \N}}.$$ 
In addition, $\Proba_p(0 \in \mathcal{R}) = \prod_{k \geq 1}(1-(1-p)^k)$, which is positive for $p \in (0,1]$.
By ergodicity, $\inf\mathcal{R} = - \infty$ and $\sup\mathcal{R} = + \infty$.  
\end{proof}

\begin{remark}\label{remark:interpretTr}
One can give the following interpretation for an element of $\mathcal{R}$ in terms of last passage percolation via the coupling from Proposition \ref{proposition:couplingProperty}: 
$$n \in \mathcal{R} \Leftrightarrow \forall k > n, \exists i_0 = n  < i_1 < \dots < i_l = k,  \forall j \in \llbracket 1 , l \rrbracket, X_{i_{j-1}, i_j} = M, $$ where $X_{i,j}$ is the weight of the edge $(i,j)$ in the graph. In other words, $n$ is in $\mathcal{R}$ if and only if there exists a directed path made of edges with weight $M$ from vertex $n$ to any vertex $k>n$ in the graph. See Figure \ref{fig:9} for an illustration.
\end{remark}
\begin{figure}
    \centering
    \input{figures/fig9.tikz}
    \caption{Illustration of Remark \ref{remark:interpretTr} with $M=1$. For the edge weights represented on the graph, $(\xi_1,\xi_2,\xi_3) =(1,2,2)$ is a triangular word, and there is a directed path made of edges with weights $1$ from $0$ to any element of $\{1,2,3\}$.}
    \label{fig:9}
\end{figure}
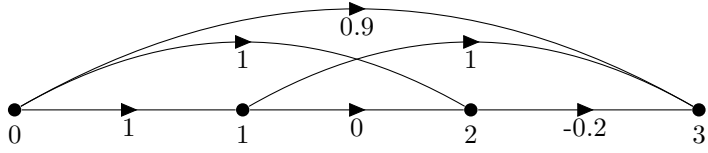

By Lemma \ref{lemma:TpastFinite}, we can enumerate the elements of $\mathcal{R}$ as $\dots < T_{-1} < T_{0} \leq 0 < T_1 < T_2 < \dots$.

\begin{remark}\label{remark:notDep}
    By Lemma \ref{lemma:keyLemmaTr} and Proposition \ref{proposition:couplingProperty}, $W_{T_{K_n}} \leq W_n < W_{T_{K_n+1}} $ for all $n \in \N$, where $(T_k)_{k \geq 1}$ enumerates the elements of $\mathcal{R}$ for the MGS defined in Proposition \ref{proposition:couplingProperty} and $K_n := \max\Set{k \in \Z}{T_k \leq n}$. In addition, Lemma \ref{lemma:keyLemmaTr} also implies that $W_{T_k} = \max_{0 \leq n \leq T_k} W_n = \lambda^{(T_k)}_1 $ for all $k \in \N$ since $W_{T_k} = 1 + \max_{0 \leq n \leq T_k-1} W_n$. 
    Since $T_k$ tends to $+\infty$ as $k$ tends to $+\infty$,  $(n^{-1}\max\{W_j, j\in\llbracket 0,n\rrbracket\})_n$ and $(n^{-1}W_n)_n$ have the same limit as $n$ tends to $+\infty$. As a consequence, \eqref{eq:convCouplingFK} is proved.
\end{remark}

\begin{remark}\label{remark:indepInit}
    By Proposition \ref{proposition:couplingProperty}, notice that the time constant $C(\nu)$ is equal to the limit of $n^{-1}(W_n-W_{T_1}) = \max_{T_1 \leq k \leq n} (\lambda^{(k)}_1 - \lambda^{(T_1)}_1)$ as $n$ tends to infinity. Since this quantity is independent of the initial configuration of the considered MGS by Lemma \ref{lemma:keyLemmaTr}, \eqref{eq:convCouplingFK} holds for any initial configuration $\lambda^{(0)}$ for $\lambda$.
\end{remark}
 
\begin{proof}[Proof of Proposition \ref{proposition:stationaryMGS}]
For $k \in \Z$, we set $Y^{(k)} = ( X^{(T_k)}, \dots, X^{(T_{k+1}-1)})$, where $ X^{(n)}= (X^{(n)}_i)_{i\in\N}. $

By definition of $T_k$, $\alpha(Y^{(k)}) \in \mathcal{T}$ for all $ k \in \Z$.
We set $y_j^{(k)} = \mathfrak{m}(\Phi_{(X^{(T_k)},\dots, X^{(j-1)})}(\theta), X^{(j)}) - \theta_1$ for all $k \in \Z$, $j \in \llbracket T_k, T_{k+1}-1\rrbracket$ and any $\theta \in \mathcal{N}$. Thanks to Lemma \ref{lemma:keyLemmaTr}, $y_j^{(k)}$ depends only on $(X^{(i)})_{i \in \llbracket T_k, T_{k+1} - 1 \rrbracket}$ and not on $\theta$. As a consequence, if $(\widetilde{\lambda}^{(n)})_{n \in \Z}$ verifies assumption \ref{enum:1} from Proposition \ref{proposition:stationaryMGS}, then the position $\mathfrak{m}(\widetilde{\lambda}^{(i-1)}, X^{(i)})$ of the particle added at time $i \in \llbracket T_0 , T_1-1 \rrbracket $ depends only on $  \widetilde{\lambda}^{(T_0-1)}_1$ and $(X^{(T_0)}, \dots , X^{(T_1-1)})$:
\begin{equation}\label{eq:proofStatMGS1}
    \mathfrak{m}(\widetilde{\lambda}^{(i-1)}, X^{(i)}) = \widetilde{\lambda}^{(T_0-1)}_1 + y_i^{(0)}.
\end{equation}
Since $T_0 \leq 0 < T_1$ and $\alpha((x^{(T_0)}, \dots , x^{(0)})) \in \mathcal{T}$, by Lemma \ref{lemma:keyLemmaTr}, $$\widetilde{\lambda}^{(0)}_1 = \max_{i \in \llbracket T_0,0\rrbracket} (\widetilde{\lambda}^{(T_0-1)}_1 + y_i^{(0)}) = \widetilde{\lambda}^{(T_0-1)}_1 +  \max_{i \in \llbracket T_0,0\rrbracket} y_i^{(0)}.$$    
If we assume that $(\widetilde{\lambda}^{(n)})_{n \in \Z}$ also verifies assumption \ref{enum:2} from Proposition \ref{proposition:stationaryMGS}, then necessarily
\begin{equation}\label{eq:proofStatMGS2}
    \widetilde{\lambda}^{(T_0-1)}_1 = - \max_{i \in \llbracket T_0,0\rrbracket} y_i^{(0)} .
\end{equation}
We set $z_i^{(0)} = y_i^{(0)}-\max_{j \in \llbracket T_0,0\rrbracket} y_j^{(0)}$ for all $i \in \llbracket T_0, T_1 -1 \rrbracket $. By \eqref{eq:proofStatMGS1} and \eqref{eq:proofStatMGS2}, for all $i\in \llbracket T_0, T_1 -1 \rrbracket$, we have  \begin{equation*}
    \mathfrak{m}(\widetilde{\lambda}^{(i-1)}, X^{(i)}) = z_i^{(0)}.
\end{equation*}
 By Remark \ref{remark:keyLemmaTr}, we notice that \begin{equation*}
     \widetilde{\lambda}^{(0)} = \widetilde{\lambda}^{(T_0 - 1)} + \sum_{i\in \llbracket T_0,0\rrbracket} \delta_{\widetilde{\lambda}^{(T_0 - 1)}_1 + y_i^{(0)}} =  \widetilde{\lambda}^{(T_0 - 1)} + \sum_{i\in \llbracket T_0,0\rrbracket} \delta_{ z_i^{(0)}}.
\end{equation*}
For $i\in \llbracket T_0, T_1 - 1 \rrbracket$, $y_i^{(0)} \geq M $  by Lemma \ref{lemma:keyLemmaTr}. Then $z_i^{(0)} \geq \widetilde{\lambda}^{(T_0 - 1)}_1 + M$. In addition, all the particles in the configuration $\widetilde{\lambda}^{(T_0 - 1)}$ have their positions in $(-\infty, \widetilde{\lambda}^{(T_0 - 1)}_1]$. Therefore, the positions of all the particles in $\widetilde{\lambda}^{(0)}$ whose positions are in $(-M,0]$ are all the $z_i^{(0)}$ for $i \in \llbracket T_{0}, 0\rrbracket $ such that $ z_i^{(0)} > -M$.

Now, we can iterate this method in order to obtain all the particles positions in $(-2M,0]$ of the configuration $\lambda^{(0)}$, and then in $(kM,0]$ for all $k \leq -1$. Let us consider an integer $k \leq -1$. For all $i \in \llbracket T_{k} , T_{k+1}-1 \rrbracket$, 
\begin{equation}\label{eq:proofConstrucStat}
    \mathfrak{m}(\widetilde{\lambda}^{(i-1)}, X^{(i)}) = \widetilde{\lambda}^{(T_{k}-1)}_1 + y_i^{(k)},
\end{equation}
As a consequence, by Lemma \ref{lemma:keyLemmaTr},
$$ \widetilde{\lambda}^{(T_{k+1}-1)}_1 = \widetilde{\lambda}^{(T_{k}-1)}_1 + \max_{j \in \llbracket T_{k}, T_{k+1} -1 \rrbracket} y^{(k)}_j$$
Then, if we have $ \widetilde{\lambda}^{(T_{k+1}-1)}_1 $, we obtain $\widetilde{\lambda}^{(T_{k}-1)}_1$, and we can compute all $\mathfrak{m}(\widetilde{\lambda}^{(i-1)}, X^{(i)}) $ for $i \in \llbracket T_{k} , T_{k+1}-1 \rrbracket$ by \eqref{eq:proofConstrucStat}. As a consequence, by induction on $k$, we obtain $\mathfrak{m}(\widetilde{\lambda}^{(i-1)}, X^{(i)}) $ for all $ T_k \leq i \leq 0 $, and those positions contains all the particles positions in $(kM,0]$ since $ \lambda_1^{(T_{i})}-\lambda_1^{(T_{i}-1)} \geq M $ for all $i \leq 0$.
Therefore, $\widetilde{\lambda}^{(0)}$ is almost surely uniquely defined by assumptions \ref{enum:1} and \ref{enum:2}. In addition, we notice that by construction, $F_{\widetilde{\lambda}^{(0)}_1}= \widetilde{\lambda}^{(0)}$ is $\mathcal{F}_0$-measurable.

Now that we have constructed $\widetilde{\lambda}^{(0)}$, for all $n \geq 1$, $\widetilde{\lambda}^{(n)}$ is obtained by iterating $n$ times assumption \ref{enum:1}. By construction, it is clear that $F_{\widetilde{\lambda}^{(n)}_1}(\widetilde{\lambda}^{(n)})$ is $\mathcal{F}_n$-measurable for all $n \geq 1$.

In order to get $\widetilde{\lambda}^{(n)}$ for $n \leq -1$, it suffices to remove the particles with positions $ \mathfrak{m}(\widetilde{\lambda}^{(i-1)}, X^{(i)})$ from $\lambda^{(0)} $ for all $i \in \llbracket n+1, 0 \rrbracket$ since the process should verify assumption \ref{enum:1}. Since we managed to compute those positions when we constructed $\widetilde{\lambda}^{(0)}$, we obtain $\widetilde{\lambda}^{(n)}$. Another way to construct $\widetilde{\lambda}^{(n)}$ for such $n$ would be to use the same method as for $\widetilde{\lambda}^{(0)}$ with the sequence $(X^{(k)})_{k \leq n}$. Since the process is almost surely uniquely defined by construction, we would almost surely obtain the same configuration. With this second method, it is clear that $F_{\widetilde{\lambda}^{(n)}_1}(\widetilde{\lambda}^{(n)})$ is $\mathcal{F}_n$-measurable, which concludes the proof.
\end{proof}

Now that we have constructed the stationary version $(\widetilde{\lambda}^{(n)})_{n \in \Z}$ of the MGS, we can couple it with the standard MGS $(\lambda^{(n)})_{n \in \Z_+}$ using the same renovation events. 

\begin{proposition}[Coupling property]\label{proposition:couplingStat}
We consider $(\widetilde{\lambda}^{(n)})_{n \in \Z}$ the stationary version of the MGS introduced in Proposition \ref{proposition:stationaryMGS} with  weights $(X^{(i)}_j)_{i \in \Z, j \in \N}$. We consider a configuration $\lambda^{(0)} \in \mathcal{N}$ and we set $ \lambda^{(n)} = \Phi_{X^{(n)}}(\lambda^{(n-1)})$ for all $n\geq 1$. Then, $(\lambda^{(n)})_{n \geq 0}$ is a standard MGS. For all $n \geq T_1$, \begin{equation}\label{eq:couplingStat}
    \mathfrak{m}(\lambda^{(n-1)}, X^{(n)}) - \mathfrak{m}(\widetilde{\lambda}^{(n-1)}, X^{(n)}) = \lambda^{(T_1 - 1)}_1 - \widetilde{\lambda}^{(T_1 - 1)}_1.
\end{equation}
In addition, $\lambda_1^{(n)} -  \widetilde{\lambda}_1^{(n)} = \lambda^{(T_1 - 1)}_1 - \widetilde{\lambda}^{(T_1 - 1)}_1$ for all $n \geq T_1$. 
\end{proposition}

In what follows, we call \emph{front position} the position of a particle with maximal position in a configuration.

\begin{proof}[Proof of Proposition \ref{proposition:couplingStat}]
By definition of $T_1$, $\alpha((x^{(T_1)},\dots, x^{(n)})) \in \mathcal{T}$ for all $n \geq T_1$. As a consequence of Lemma \ref{lemma:keyLemmaTr}, $\mathfrak{m}(\Phi_{(X^{(1)}, \dots , X^{(n-1)})}(\theta), X^{(n)}) - \theta_1 $ does not depends on the configuration $\theta \in \mathcal{N}$.
Then, for all $n \geq T_1$, \begin{equation*}
     \mathfrak{m}(\lambda^{(n-1)}, X^{(n)}) -  \lambda^{(T_k - 1)}_1 = \mathfrak{m}(\widetilde{\lambda}^{(n-1)}, X^{(n)}) -  \widetilde{\lambda}^{(T_k - 1)}_1.
\end{equation*} 
which completes the proof of \eqref{eq:couplingStat}.
By Lemma \ref{lemma:keyLemmaTr}, we also know that $\mathfrak{m}(\Phi_{(X^{(1)}, \dots , X^{(n-1)})}(\theta), X^{(n)}) - \theta_1 \geq M > 0 $ for all $n \geq T_1$. 
Therefore, all the particles added at a time $n \geq T_1$ in the two versions of the MGS have positions greater than the front positions at time $T_1-1$ in both MGS. In addition, all the particles added at time $n < T_1$ have positions less or equal to the front position at time $T_1-1$ for both MGS.
Then, for all $n \geq T_1$, $\lambda_1^{(n)} - \lambda_1^{(T_1 -1)} =\widetilde{\lambda}_1^{(n)} - \widetilde{\lambda}_1^{(T_1 -1)} = \max_{i \in \llbracket T_1, n \rrbracket} \mathfrak{m}(\Phi_{(X^{(1)}, \dots , X^{(n-1)})}(\theta), X^{(n)}) - \theta_1$ where $\theta$ is any element of $ \mathcal{N}$.
\end{proof}

\subsection{An abstract formula for the time constant for last passage percolation}
\label{subsec:formulaC}

In this section, we give a formula for $C((1-p)\mu + p\delta_M)$ in terms of the renovation events given in Subsection \ref{subsec:statMGS}.

For $\alpha \in \mathcal{A}$, we denote by $H(\alpha) = \sum_{i=1}^{|\alpha|} (\alpha_i - 1)$ the \emph{height} of the word $\alpha$. 
We set $\mathcal{T}_m$ 
 to be the set of all the words in $\mathcal{T}$ with no strict suffix in $\mathcal{T}$. More precisely, $\alpha \in \mathcal{T}_m$ if and only if $ \alpha \in \mathcal{T} $ and for all $i\in\llbracket 2 , |\alpha|\rrbracket$, $(\alpha_i, \dots, \alpha_{|\alpha|}) \notin \mathcal{T}$. The elements of $\mathcal{T}_m$ are called \emph{minimal triangular words}.

For $\beta \in \mathcal{T}$ and $x = (x^{(i)}_j)_{i \in \llbracket 1 , l \rrbracket, j \in \llbracket 1 , \beta_i-1 \rrbracket}$ with $x^{(i)}_j \in [-\infty, M)$ for all $(i,j)$, let $\send{M}{\beta}{x}$ be the difference of front positions between steps $l-1$ and $l$ when applying a weight sequence $x'\in (\mathcal{W}_{\leq M})^l$ such that $\alpha(x') = \beta$ and  $x'^{(i)}_{j}= x_{j}^{(i)}$ for all $1 \leq i \leq l$ and $ 1 \leq j \leq \beta_i - 1$ to any configuration $\lambda \in \mathcal{N}$. By Lemma \ref{lemma:keyLemmaTr}, this quantity depends only on $ \beta$ 
and $(x^{(i)}_j)_{i \in \llbracket 1 , l \rrbracket, j \in \llbracket 1 , \beta_i-1 \rrbracket} $.  It is clear that $\send{M}{\beta}{x} \in [0,M]$. 

More precisely, Lemma \ref{lemma:keyLemmaTr} asserts that $\mathfrak{m}(\Phi_{(x'^{(1)},\dots, x'^{(l-l)})}(\lambda), x'^{(l)})-\lambda_1 $ depends only on $ \beta $ and $(x^{(i)}_j)_{i \in \llbracket 1 , l \rrbracket, j \in \llbracket 1 , \beta_i-1 \rrbracket} $. Then, 
\begin{equation}\label{eq:send}
\send{M}{\beta}{ (x^{(i)}_j)_{i \in \llbracket 1 , l \rrbracket, j \in \llbracket 1 , \beta_i-1 \rrbracket}}= \max_{i \in \llbracket 1,l \rrbracket}y_i - \max_{i \in \llbracket 1, l-1 \rrbracket}y_i,
\end{equation}  where $y_i =\mathfrak{m}(\Phi_{(x'^{(1)},\dots, x'^{(i-l)})}(\lambda), x'^{(i)})-\lambda_1 $ for all $i\in\llbracket1 , l \rrbracket$, any configuration $\lambda \in \mathcal{N}$ and any $x' \in (\mathcal{W}_{\leq M})^l$ such that $\alpha(x') = \beta$ and $x'^{(i)}_{j}= x_{j}^{(i)}$ for all $1 \leq i \leq l$ and $ 1 \leq j \leq \beta_i - 1$.

\begin{proposition}\label{proposition:formuleMGS}
For all $p\in (0,1]$, $M\in \R^*_+$ and $\mu \in \Mof{[-\infty , M)}$, 
\begin{equation*}
    C((1-p)\mu + p\delta_M) = \sum_{\beta \in \mathcal{T}_m}p^{|\beta|}(1-p)^{H(\beta)}\int \send{M}{\beta}{ (x_{j}^{(i)})_{i,j}} \prod_{\substack{ 1 \leq i \leq |\beta|\\ 1 \leq j \leq \beta_i -1}} d\mu(x^{(i)}_{j}).
\end{equation*}
\end{proposition}

\begin{proof}
By \eqref{eq:convCouplingFK}, for $(\lambda^{(n)})_{n \in \Z_+}$ an MGS with a starting configuration $ \lambda^{(0)}$, we know that $ (n^{-1}\lambda_1^{(n)})$ converges to $ C(\nu)$ almost surely and in $L^1$ as $n$ tends to infinity. Consider  $\widetilde{\lambda}$ the stationary MGS introduced in Proposition \ref{proposition:stationaryMGS}.
Since $n^{-1}(\lambda^{(T_1 - 1)}_1 - \widetilde{\lambda}^{(T_1 - 1)}_1)$ converges almost surely and $L^1$ to $0$ as $n$ tends to infinity by Proposition \ref{proposition:couplingStat}, we also have  
\begin{align*}
    \frac{\widetilde{\lambda}_1^{(n)}}{n} \xrightarrow[n \rightarrow \infty]{\text{a.s, } L^1} C(\nu).
\end{align*}  
Then, by $L^1$ convergence,
\begin{align}\label{eq:proofLimCnu}
    \frac{1}{n}\E\left[\widetilde{\lambda}_1^{(n)}\right] \xrightarrow[n \rightarrow \infty]{} C(\nu).
\end{align} 
By stationarity of $(\widetilde{\lambda}^{(n)})_{n \in \Z}$ (see Remark \ref{remark:stat}), $\E[\widetilde{\lambda}^{(k+1)}_1- \widetilde{\lambda}^{(k)}_1] = \E[\widetilde{\lambda}^{(1)}_1- \widetilde{\lambda}^{(0)}_1] = \E[\widetilde{\lambda}^{(1)}_1] $ for all $k \in \Z$. Therefore, $$ \E\left[\widetilde{\lambda}_1^{(n)}\right] = \sum_{k = 0}^{n-1} \E\left[\widetilde{\lambda}^{(k+1)}_1- \widetilde{\lambda}^{(k)}_1\right] = n \E\left[\widetilde{\lambda}^{(1)}_1\right]. $$
As a consequence, by \eqref{eq:proofLimCnu}, $  C(\nu) =\E\left[ \widetilde{\lambda}^{(1)}_1\right].$
Now, we rewrite this formula in terms of triangular words and weights sequences. We set $$\Tpast = \inf\Set{t \in \N}{( X^{(-t+2)}, \dots ,X^{(1)} )\in\mathcal{T}}.$$ By definition of $\Tpast$, we have $\alpha( (X^{(2-\Tpast)}, \dots, X^{(1)})) \in \mathcal{T}_m$. Then, by the law of total probability 
\begin{align}\label{eq:proofFormulaCnu3}
    C(\nu) &= \sum_{\beta \in \mathcal{T}_m}    \Proba( \alpha((X^{(2-\Tpast)}, \dots, X^{(1)})) = \beta) \E\left[  \widetilde{\lambda}^{(1)}_1 \mid \alpha((X^{(2-\Tpast)}, \dots, X^{(1)})) = \beta  \right] \notag \\
    &= \sum_{\beta \in \mathcal{T}_m}   p^{|\beta|}(1-p)^{H(\beta)} \E\left[  \widetilde{\lambda}^{(1)}_1 \mid \alpha((X^{(2-\Tpast)}, \dots, X^{(1)})) = \beta  \right].
\end{align}

Conditionally to the event $\{\alpha((X^{(2-\Tpast)}, \dots, X^{(1)})) = \beta\}$, $(X_j^{(i)})_{i \in \llbracket 1 , |\beta| \rrbracket, j \in \llbracket 1 , \beta_i -1\rrbracket}$ are i.i.d. with distribution $\mu$. Therefore, by definition of $s^{\text{end}}_M$,
\begin{equation}\label{eq:proofFormulaCnu2}
    \E\left[  \widetilde{\lambda}^{(1)}_1 \mid \alpha((X^{(2-\Tpast)}, \dots, X^{(1)})) = \beta  \right] = \int \send{M}{\beta}{(x^{(i)}_{j})_{i,j}} 
    \prod_{\substack{ 1 \leq i \leq |\beta|\\ 1 \leq j \leq \beta_i -1}} d\mu(x^{(i)}_{j}).
\end{equation} 

Putting everything together, we obtain the formula from Proposition \ref{proposition:formuleMGS}, which completes the proof of the proposition.
\end{proof}

\section{Regularity properties of the time constant}
\label{sec:regularity}

The four subsections are devoted respectively to the proofs of Theorem \ref{theorem:analyticCpmu}, Theorem \ref{theorem:analyticityFiniteSupport}, Theorem  \ref{theorem:continuity} and Theorem \ref{theorem:strictIncreasing}.

\subsection{Analyticity in \texorpdfstring{$p$}{TEXT} for measures of the form \texorpdfstring{$(1-p)\mu + p\delta_M$}{TEXT}}
\label{subsec:proofAnaliticCpmu}

With the results from Section \ref{sec:stationarity}, we now prove Theorem \ref{theorem:analyticCpmu} regarding the analyticity of $p \mapsto C((1-p)\mu + p\delta_M)$ on $(0,1]$.

\begin{proof}[Proof of Theorem \ref{theorem:analyticCpmu}]
We fix $M\in \R^*_+$. We denote by $ D(z_0,r) := \Set{z \in \CC}{|z-z_0| < r}$ the open disk of radius $r$ and center $z_0$.
By Proposition \ref{proposition:formuleMGS}, it suffices to show that for all $q_0 \in [0,1)$, there exists a $\delta > 0 $ such that for all $z \in D(0,\delta )$, 
\begin{align*}
    \sum_{\beta \in  \mathcal{T}_m} |q_0 + z|^{H(\beta)}|1-q_0-z|^{|\beta|} \int \left|\send{M}{\beta}{(x^{(i)}_{j})_{i,j}}\right| \prod_{\substack{ 1 \leq i \leq |\beta|\\ 1 \leq j \leq \beta_i -1}} d\mu(x^{(i)}_{j}).
\end{align*}

We consider $q_0 \in [0,1)$ and $z \in \CC$. Since $\mu$ is a probability distribution and $\send{M}{\beta}{(x^{(i)}_{j})_{i,j}}$ takes values in $[0,M]$, it suffices to show that for some $\delta >0$ and all $z \in D(0, \delta)$,
\begin{align}\label{eqToCheckAnalyticity}
    \sum_{\beta \in  \mathcal{T}_m} |q_0 + z|^{H(\beta)}|1-q_0-z|^{|\beta|}   < \infty. 
\end{align}
By the triangle inequality, we have 
\begin{align*}
    \sum_{\beta \in  \mathcal{T}_m} |q_0 + z|^{H(\beta)}|1-q_0-z|^{|\beta|} &\leq \sum_{\beta \in  \mathcal{T}_m} (q_0 + |z|)^{H(\beta)}(1-q_0+|z|)^{|\beta|}.
\end{align*}

For $z$ such that $1-q_0 - |z| > 0$, we have 
\begin{align*}
    \sum_{\beta \in  \mathcal{T}_m} |q_0 + z|^{H(\beta)}|1-q_0-z|^{|\beta|} 
    &\leq \sum_{\beta \in \mathcal{T}_m} (q_0 + |z|)^{H(\beta)}(1-q_0-|z|)^{|\beta|}  \left(\frac{1-q_0+|z|}{1-q_0-|z|}\right)^{|\beta|}.
\end{align*}

For $p \in (0,1]$ and $\mu$ a probability distribution such that $\mu([-\infty,M))=1$, we denote by $\E_{p,\mu}$ the expectation corresponding to a probability space where the weights $X_j^{(i)}$ have distribution $ (1-p) \mu +  p \delta_M$. 
For $r>1$, we notice that  
\begin{align*}
     \E_{p,\mu} \left[ r^{\Tpast} \right]  = \sum_{\beta \in \mathcal{T}_m} (rp)^{|\beta|}(1-p)^{H(\beta)}.
\end{align*} 
by the law of total probability applied to $\alpha((X^{(2-\Tpast)}, \dots , X^{(1)}))$. Therefore,
\begin{align}\label{step2analyticity}
    \sum_{\beta \in  \mathcal{T}_m} |q_0 + z|^{H(\beta)}|1-q_0-z|^{|\beta|} 
    &\leq \E_{1-q_0 - |z|,\mu}\left[ \left(\frac{1-q_0+|z|}{1-q_0-|z|}\right)^{\Tpast} \right].
\end{align}

If we show that the moment generating function of $\Tpast$ is finite on an open neighborhood of $1$, then we can show that both quantities in \eqref{step2analyticity} are finite for $|z|$ small enough. 
To do so, we introduce $\xi_n = \inf\{i \in \N, X^{(n)}_i = M\}$ for all $n \in \Z$. We consider the random sequence $(M_n)_{n \geq 0}$, where $M_n$ corresponds to the number of elements $i$ in $\llbracket -n+2 , 1 \rrbracket$ such that $\xi_{i} > i+n-1$. 
More explicitly, we consider $\mathcal{M}_0 = \varnothing$, and for all $n \geq 1$, $$\mathcal{M}_n = \Set{i \in \llbracket -n+2, 1 \rrbracket}{ \xi_{i} > i+n-1}. $$ 
For $n \in \Z_+$, we set $M_n = \#\mathcal{M}_n$.
By definition of $\Tpast$, we notice that $$\Tpast = \inf \Set{n \in \N}{ M_n = 0}.$$

See Figure \ref{fig:10} for an illustration. 
Equivalently, by definition of $(\xi_i)_{i \in \Z}$, for all $n \in \N$,
$$\mathcal{M}_n = \Set{i \in \llbracket -n+2, 1 \rrbracket}{ \forall j \in \llbracket 1, i+n-1\rrbracket, X_{j}^{(i)} < M}.$$ 

We also notice that for all $n \geq 1$, $$ \mathcal{M}_n = \{i \in \{-n+2\}\cup\mathcal{M}_{n-1}, X^{(i)}_{i+n-1} < M\} .$$ Since $\mathcal{M}_{n-1}$ depends only on $(X^{(i)}_{j})_{ i \in \llbracket -n+1 , 1 \rrbracket, j \in \llbracket 1 , i+n-2 \rrbracket}$ which is independent of $(X^{(i)}_{i+n-1})_{i \in \llbracket -n+2,1\rrbracket}$, the conditional distribution of $M_n$ given $M_{n-1}$ is a binomial distribution $\Bin(M_{n-1}+1, 1-p)$. 

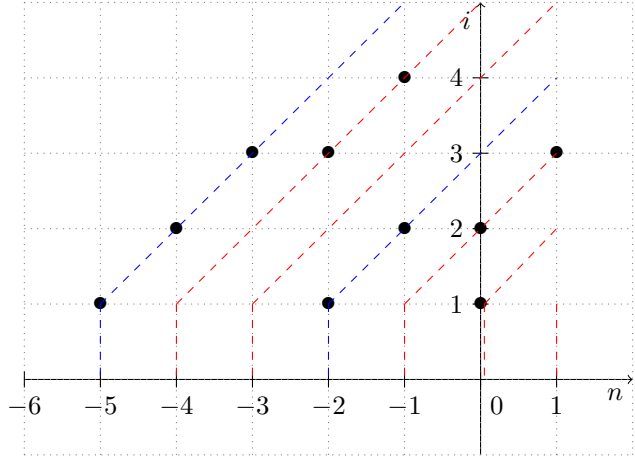
\begin{figure}
    \centering
    \input{figures/fig10.tikz}
    \caption{A dot is represented at position $(n,i)$ if and only if $X^{(n)}_i=M$. Notice that $(\xi_k, \dots , \xi_0, \xi_1)$ is triangular if and only if, for all $i \in \llbracket k,1\rrbracket$, there is at least one dot with abscissa $i$ below or on the dashed line passing through $(k,1)$. 
    For $n \in \N$, $\mathcal{M}_n$ is the set of all $i\in \llbracket 2-n,1\rrbracket$ for which there is no dot with abscissa $i$ below or on the dashed line passing through $(2-n,1)$. 
    Here, $\mathcal{M}_1 = \mathcal{M}_2= \{1\}$, $\mathcal{M}_{3} = \{-1\}$, $ \mathcal{M}_{4} =  \mathcal{M}_{7}  = \varnothing$, $\mathcal{M}_{5} = \{-4,-3 \}$ 
    and $\mathcal{M}_{6} = \{-4\}$. 
    Then, $\Tpast= 4 $.} 
    \label{fig:10}
\end{figure}

We notice that $M_n $ has the same law as a Galton-Watson process with immigration, with distribution $\Ber(1-p)$ for the number of offspring of one particle and with distribution $\Ber(1-p)$ for the immigration at each step. In \cite{MR0300351}, Zubkov shows the following result :
\begin{theorem}[{\cite[Theorem 1]{MR0300351}}]\label{theorem:Zubkov}
For any $ p \in (0,1)$, there exists a constant $c$ such that for all $k > 0$,  $$\Proba ( \Tpast > k ) \sim c (r_p)^{-(k+1)} ,$$ where $r_p$ is the only element $r$ of $(1,\frac{1}{1-p})$ such that $\sum_{k=0}^{\infty}(qr)^{k+1}\prod_{n=1}^{k} (1-q^n) = 1 $ with $q:=1-p$.
\end{theorem}

As a consequence of Theorem \ref{theorem:Zubkov}, $\Tpast$ has finite exponential moments. More precisely:

\begin{lemma}\label{lemma:TExpMoment}
For all $p \in (0,1]$, there exists $\rho_{p} >1$ such that for all $p' \in [p,1]$ and $\mu \in \Mof{[-\infty,M)}$,
\[\E_{p',\mu}[ \rho_{p}^{\Tpast}]< \infty.\]
\end{lemma}
\begin{proof}[Proof of Lemma \ref{lemma:TExpMoment}]
For $p=1$, the result is trivial since $\Tpast=1$ almost surely. Now, we assume that $p\in(0,1)$. By Theorem \ref{theorem:Zubkov}, for any $1 < \rho_p < r_p$, 
$\E_{p,\mu}[ \rho_{p}^{\Tpast}]< \infty$. In addition, $\Tpast$ is a non-increasing function of $p$ for the trivial coupling. Therefore, $\E_{p',\mu}[ \rho_{p}^{\Tpast}]< \infty$ for all $p' \in [p,1]$, since $r_p > 1$. 
\end{proof}

Now, we consider $1<r<r_{\frac{1-q_0}{2}}$ from Lemma \ref{lemma:TExpMoment}. Setting $\delta= \min ( \frac{1-q_0}{2}, \frac{(1-q_0)(r-1)}{r+1} )$ and $ z \in D(0, \delta)$, we have $\E_{1-q_0-|z|,\mu}[ r^{\Tpast}]< \infty $ by Lemma \ref{lemma:TExpMoment} since $|z| < \frac{1-q_0}{2}$.
Furthermore, $\frac{1-q_0+|z|}{1-q_0-|z|} < r$ because $|z| < \frac{(1-q_0)(r-1)}{r+1} $. Then, we have $\eqref{eqToCheckAnalyticity}$  by \eqref{step2analyticity}, which concludes the proof of Theorem \ref{theorem:analyticCpmu}.
\end{proof}

\begin{remark}\label{remark:radius}
There is no easy way to compute $r_p$ explicitly, but we can approximate its value. We can obtain a lower bound for the radius of convergence of $p \mapsto C((1-p)\mu + p \delta_M)$. This lower bound is valid for all $\mu \in \Mof{[-\infty,M)}$ 
since $\Tpast$ does not depend on $\mu$. Let us consider $\varepsilon \in (0,1), q_0 \in [0,1)$ and $r = r_{\varepsilon(1-q_0)} $. One can show by similar computations as above that $\delta=\min ( (1-q_0)(1- \varepsilon), \frac{(1-q_0)(r-1)}{r+1} ) $ is also a lower bound for the radius of convergence of $p \mapsto C((1-p)\mu + p \delta_M)$ around $1-q_0$. For $q_0=0$, we can numerically optimize this lower bound in $\varepsilon$. We obtain that the radius of convergence around $1$ is lower bounded by $0.298167$. This lower bound does not seem to be optimal according to numerical simulations, but it is a slight improvement over the previously known lower bound $\frac{\sqrt{2}-1}{2} \simeq 0,20710678$ from \cite{MaRa2016} for the case where $\mu=\delta_{-\infty}$.  
\end{remark}

\subsection{Analyticity for measures with \texorpdfstring{$N$}{TEXT} atoms}
\label{subsec:AnalyticitySeveralVar}

We first simplify the formula from Proposition \ref{proposition:formuleMGS} in the specific case where $\nu$ has finite support:

\begin{proposition}
We consider $N$ weights $a_1  > a_2 > \dots > a_N \geq - \infty$ such that $a_1 > 0$, and $p_1, \dots, p_N \in [0,1]$ such that $p_1 + \dots + p_N =1$ and $p_1 >0$. Then,
\begin{equation}\label{eq:CmultFormula}
    C\left(\sum_{i=1}^N p_i \delta_{a_i}\right) =  \sum_{\beta \in \mathcal{T}_m}   
    \sum_{x \in E_{\beta}} \send{a_1}{\beta}{x}   p_1^{|\beta|} p_2^{c_{2,x}} \dots p_N^{c_{N,x}} ,
\end{equation}
where $E_{\beta}  = \Set{(x_j^{(i)})_{i\in\llbracket 1 , |\beta| \rrbracket , j \in \llbracket 1 , \beta_i - 1  \rrbracket}}{ x_j^{(i)} \in \{a_2, \dots a_N\}}  $ and  $c_{k,x} = |\Set{(i,j) \in \llbracket 1 , |\beta| \rrbracket\times \llbracket 1 , \beta_i - 1  \rrbracket}{ x_{j}^{(i)}=a_k}|$ for all $k \in \llbracket 1 , N\rrbracket$ and $x \in E_{\beta}$.
\end{proposition}
\begin{proof}
We apply Proposition \ref{proposition:formuleMGS} with $p = p_1$ and $\mu = \frac{1}{1-p_1} \sum_{i=2}^N p_i \delta_{a_i}$. We obtain 
\begin{align*}
    C\left(\sum_{i=1}^N p_i \delta_{a_i}\right) &= \sum_{\beta \in \mathcal{T}_m} p_1^{|\beta|}(1-p_1)^{H(\beta)} \int \send{a_1}{\beta}{(x_{j}^{(i)})_{i,j}} 
    \prod_{\substack{ 1 \leq i \leq |\beta|\\ 1 \leq j \leq \beta_i -1}} d\mu(x^{(i)}_{j})\\
    &= \sum_{\beta \in \mathcal{T}_m} \sum_{x \in E_{\beta}} p_1^{|\beta|}(1-p_1)^{H(\beta)}  \send{a_1}{\beta}{x} \prod_{i=1}^{|\beta|} \prod_{j=1}^{\beta_i - 1} \mu(\{x_j^{(i)}\})\\
    &= \sum_{\beta \in \mathcal{T}_m}\sum_{x \in E_{\beta}} p_1^{|\beta|}(1-p_1)^{H(\beta)}   \send{a_1}{\beta}{x} \left(\prod_{i=1}^{|\beta|} \prod_{j=1}^{\beta_i - 1} (1-p_1)^{-1}p_2^{I_{i,j,2}} \dots p_N^{I_{i,j,N}}\right) \\
    &= \sum_{\beta \in \mathcal{T}_m} \sum_{x \in E_{\beta}}  \send{a_1}{\beta}{x} \cdot p_1^{|\beta|}   p_2^{c_{2,x}} \dots p_N^{c_{N,x}}, 
\end{align*}
where $I_{i,j,k} := \mathds{1}_{x_j^{(i)} = a_k}$.
\end{proof}

Let us now prove Theorem \ref{theorem:analyticityFiniteSupport}:
\begin{proof}[Proof of Theorem \ref{theorem:analyticityFiniteSupport}]
We consider $N$ weights $a_1 > a_2 > \dots > a_N \geq - \infty$ with $a_1 > 0$. By Hartogs' theorem on separate holomorphicity \cite{MR1846625}, 
if $f$ is a function defined on an open set $U\subseteq \CC^n$ such that, for all $i \in \llbracket 1 , n \rrbracket$, $z_i \mapsto f(z_1, \dots , z_n)$ is analytic when the other coordinates $z_k$ for $k \neq i$ are fixed, then $f$ is analytic on $U$.  Therefore, it suffices to show that for every $k \in \llbracket 2 , N\rrbracket$, $$p_k \mapsto C\left(\sum_{i=2}^N p_i \delta_{a_i} + (1-p_2-\dots-p_N)\delta_{a_1} \right)$$ is analytic on $\{ p_k \in [0,1], 0 \leq p_2 + \dots +p_N < 1 \}$ for all $p_2, \dots, p_{k-1}, p_{k+1}, \dots, p_N \in [0,1]$ such that $p_2+ \dots + p_{k-1} + p_{k+1}+ \dots + p_N < 1$. 

Now, we fix $k \in \llbracket 2 , N\rrbracket$. By \eqref{eq:CmultFormula}, it suffices to show that for all $(p_2, \dots, p_N)\in [0,1]^N$ such that $0 \leq p_2 + \dots + p_N < 1$, 
there exists $\delta >0$ such that for all $z \in D(0,\delta)$, 
$$(\star) := \sum_{\alpha \in \mathcal{T}_m} \sum_{x \in E_{\alpha}} |1-p_2-\dots- (p_k+z) - \dots - p_N |^{|\alpha|} \cdot  p_2^{c_{2,x}} \dots  p_{k-1}^{c_{k-1,x}} |p_k + z|^{c_{k,x}}\cdot p_{k+1}^{c_{k+1,x}} \dots  p_N^{c_{N,x}} < \infty $$ since $0 \leq s^{end}_{a_1} \leq a_1$.
We set $p_1 = 1-p_2 - \dots - p_N$. By the triangle inequality, we obtain the following upper bound \begin{align*}
    (\star) &\leq \sum_{\alpha \in \mathcal{T}_m} \sum_{x \in E_{\alpha}} (p_1 + |z| )^{|\alpha|} \cdot  p_2^{c_{2,x}} \dots  p_{k-1}^{c_{k-1,x}} (p_k + |z|)^{c_{k,x}}\cdot p_{k+1}^{c_{k+1,x}} \dots  p_N^{c_{N,x}}  \\
    &= \sum_{\alpha \in \mathcal{T}_m} \sum_{x \in E_{\alpha}} \left( \frac{p_1+|z|}{p_1-|z|} \right)^{|\alpha|} (p_1 - |z| )^{|\alpha|} \cdot  p_2^{c_{2,x}} \dots  p_{k-1}^{c_{k-1,x}} (p_k + |z|)^{c_{k,x}}\cdot p_{k+1}^{c_{k+1,x}} \dots  p_N^{c_{N,x}}  \\
\end{align*}
We take $\delta < \min(1-p_k,1-p_1) $ so that $p_k + |z|\in [0,1[$ and $1-p_1-|z|>0$. If we denote by $\E_{p, \mu}$ the expectation corresponding to the probability space where the weights have distribution $(1-p)\mu + p\delta_1$, then $$ (\star)\leq \E_{p_1-|z|,\sum_{i=2}^N \frac{p_i}{1-p_1} \delta_{a_i}}\left[\left( \frac{p_1+|z|}{p_1-|z|} \right)^{\Tpast}\right],$$
where $\Tpast$ is the length of the first triangular word in the past for a word of i.i.d. random variables with distribution $\mathcal{G}(p_1-|z|)$. This expectation is exactly the same as in \eqref{step2analyticity}, and its finiteness has already been proved by Lemma \ref{lemma:TExpMoment} for $|z|\leq \delta'$ for some $\delta'$. Then, with $\delta < \min(1-p_k,1-p_1, \delta') $, $(\star) < +\infty$. 

Therefore, $p_k \mapsto C\left(\sum_{i=2}^N p_i \delta_{a_i} + (1-p_2-\dots-p_N)\delta_1 \right)$ is analytic on $\{ p_k \in [0,1], 0 \leq p_2 + \dots +p_N < 1 \}$.
\end{proof}

\subsection{Continuity}
\label{subsec:continuity}

\begin{proof}[Proof of Theorem \ref{theorem:continuity}]

We proceed by successive approximations to prove the theorem. Firstly, we show that it suffices to prove the result on the set of measures $\nu$ with essential supremum equal to a fixed $M$ and with $\nu((M-\varepsilon,M))=0$ for some $\varepsilon>0$. Secondly, we use Theorem \ref{theorem:analyticCpmu} and we prove the continuity of $ \mu \mapsto C((1-p)\mu + p\delta_{M_{\nu}})$ via Proposition \ref{proposition:formuleMGS} in order to prove the continuity of $ \nu \mapsto C(\nu)$ on this class of measures.

Let us consider $\nu_n,\nu \in \mathcal{M}_1$ such that $d(\nu_n,\nu)$ converges to $0$ as $n$ tends to $+\infty$. In other terms, $(\nu_n)_n$ weakly converges to $\nu$ and $M_{\nu_n}$ converges to $M_{\nu}$ as $n$ tends to $+\infty$. We want to show that $C(\nu_n)$ converges to $C(\nu)$ as $n$ tends to $+\infty$.

For $\varepsilon > 0$, we set $\nu^{(\varepsilon)} = (1-p^{(\varepsilon)}) \mu^{(\varepsilon)} + p^{(\varepsilon)} \delta_{M_{\nu}}$ where $p^{(\varepsilon)}=\nu([M_{\nu}-\varepsilon,M_{\nu}])$ and 
$\mu^{(\varepsilon)}= \frac{\nu([-\infty, M_{\nu}-\varepsilon)\cap \bullet)}{1-p^{(\varepsilon)}}.$ 
Similarly, we set $\nu_n^{(\varepsilon)} = (1-p_n^{(\varepsilon)}) \mu_n^{(\varepsilon)} + p_n^{(\varepsilon)} \delta_{M_{\nu}}$ where $p_n^{(\varepsilon)}=\nu_n([M_{\nu}-\varepsilon,M_{\nu_n}])$ and 
$\mu_n^{(\varepsilon)}= \frac{\nu_n([-\infty, M_{\nu}-\varepsilon)\cap \bullet)}{1-p_n^{(\varepsilon)}}.$ 
Notice that the probability distribution $\mu^{(\varepsilon)}$ is well-defined if and only if $p^{(\varepsilon)}<1$. When $p^{(\varepsilon)}=1$, set $\nu^{(\varepsilon)} = \delta_{M_{\nu}}$. Similarly,  $\mu_n^{(\varepsilon)}$ is not well-defined when $p_n^{(\varepsilon)}=1$, and we set $\nu_n^{(\varepsilon)} = \delta_{M_{\nu}}$ in that case. 

Since a probability distribution has a set of atoms which is at most countable, we can assume that $\nu(\{M_{\nu}-\varepsilon\})=0$ for $\varepsilon$ as small as we need. Then, $ p_n^{(\varepsilon)}$ converges to $p^{(\varepsilon)}$ as $n$ tends to $+\infty$ by Portmanteau theorem (c.f \cite{MR1700749}) for such $\varepsilon$. 
 Then, $\mu_n^{(\varepsilon)}$ is also well-defined for $n$ large enough when $p^{(\varepsilon)} < 1$ by convergence of $(p_n^{(\varepsilon)})_n$ to $p^{(\varepsilon)}$.  
Similarly, it is easy to show that $(\mu_n^{(\varepsilon)})_n$ weakly converges to $\mu^{(\varepsilon)}$ 
by the Portmanteau theorem when $p^{(\varepsilon)} < 1$. As a consequence, $(\nu_n^{(\varepsilon)})_n$ weakly converges to $\nu^{(\varepsilon)}$.

By the triangle inequality, for all $n \in \N$ and $\varepsilon>0$, \begin{equation}\label{eq:proofContinuity0}
 |C(\nu_n)-C(\nu)|  \leq  |C(\nu_n)-C(\nu_n^{(\varepsilon)})|+|C(\nu_n^{(\varepsilon)})-C(\nu^{(\varepsilon)})|+|C(\nu^{(\varepsilon)})-C(\nu)| .    
\end{equation}

The first and last terms in the right-hand side of \eqref{eq:proofContinuity0} are easy to control. 
We consider $(X^{(n)}_{i,j})_{i<j}$ i.i.d random variables with distribution $\nu_n$. If we set $$\widetilde{X}^{(n)}_{i,j} := M_{\nu} \mathds{1}_{X^{(n)}_{i,j}\geq M_{\nu}-\varepsilon} + X^{(n)}_{i,j}\mathds{1}_{X^{(n)}_{i,j}<M_{\nu}-\varepsilon},$$ then $(\widetilde{X}^{(n)}_{i,j})_{i<j}$ are i.i.d random variables with distribution $\nu_n^{(\varepsilon)} $. For $\pi_N = (i_{N,1}, \dots, i_{N,k_N})$ a heaviest path starting at $0$ and ending at $N$ for the weights $({X}^{(n)}_{i,j})_{i<j}$, we have 
$$C(\nu_n)- C(\nu_n^{(\varepsilon)}) \leq \liminf_{N \rightarrow \infty} \frac{1}{N}\sum_{l=1}^{k_N-1} ({X}^{(n)}_{i_{N,l},i_{N,l+1}} - \widetilde{X}^{(n)}_{i_{N,l},i_{N,l+1}})$$ by \eqref{eq:defC}. For all $i<j$, $({X}^{(n)}_{i,j} - \widetilde{X}^{(n)}_{i,j})$ is positive  
only when $M_{\nu}< X^{(n)}_{i,j} \leq M_{\nu_n}$. In this case, $\widetilde{X}^{(n)}_{i,j} = M_{\nu}$. Then, we obtain the following domination 
\begin{align}\label{ineg1approx}
   C(\nu_n)- C(\nu_n^{(\varepsilon)}) \leq \max(M_{\nu_n} - M_{\nu}, 0). 
\end{align}

Similarly, if $\pi_N' = (i'_{N,1}, \dots, i'_{N,k_N'})$  
is a heaviest path starting at $0$ and ending at $N$ for the weights $(\widetilde{X}^{(n)}_{i,j})_{i<j}$, then
$$C(\nu_n^{(\varepsilon)}) -C(\nu_n) \leq \liminf_{N \rightarrow \infty} \frac{1}{N}\sum_{l=1}^{k_N'-1} (\widetilde{X}^{(n)}_{i'_{N,l},i'_{N,l+1}} - {X}^{(n)}_{i'_{N,l},i'_{N,l+1}})$$ by \eqref{eq:defC}. Since $(\widetilde{X}^{(n)}_{i,j} - {X}^{(n)}_{i,j})$ is positive only when $M_{\nu} - \varepsilon \leq {X}^{(n)}_{i,j} <  M_{\nu}$, we obtain \begin{align}\label{ineg2approx}
     C(\nu_n^{(\varepsilon)}) -C(\nu_n) \leq M_{\nu} - (M_{\nu} - \varepsilon) \leq \varepsilon.
\end{align}
As a consequence of \eqref{ineg1approx} and \eqref{ineg2approx}, 
$    |C(\nu_n^{(\varepsilon)}) -C(\nu_n)| \leq \varepsilon + |M_{\nu_n}-M_{\nu}|.$ A
 similar reasoning shows that $0 \leq C(\nu^{(\varepsilon)}) -C(\nu) \leq \varepsilon.$ Then, $|C(\nu_n)-C(\nu)| \leq |C(\nu_n^{(\varepsilon)})-C(\nu^{(\varepsilon)})| + 2\varepsilon + |M_{\nu_n}-M_{\nu}|$ for all $n\geq 0$ by \eqref{eq:proofContinuity0}. 
By convergence of $(M_{\nu_n})_n$ to $M_{\nu}$, we obtain $$\limsup_{n \rightarrow \infty}|C(\nu_n)-C(\nu)| \leq\limsup_{n \rightarrow \infty}|C(\nu_n^{(\varepsilon)})-C(\nu^{(\varepsilon)})| + 2\varepsilon, $$

Since we can take $\varepsilon$ as small as we want, it remains to see that $|C(\nu_n^{(\varepsilon)})-C(\nu^{(\varepsilon)})| $ converges to $0$ as $n$ tends to $+\infty$. To do so, we use the following two lemmas:

\begin{lemma}\label{lemma:stepContinuity1}
For all $p \in (0,1]$ and $M>0$, $ \mu \mapsto C((1-p)\mu + p \delta_{M})$ is continuous for the L\'evy-Prokhorov metric on the space of probability distributions in $\Mof{[-\infty,M)}$.
\end{lemma}

\begin{lemma}\label{lemma:stepContinuity2}
For all $M > 0$ and $p \in (0,1]$, there exist $L \in \R^+$ and  $p_- < p_+ \in \R_+^*$ such that $ p_- < p < p_+$ and such that for all $\mu \in \Mof{[-\infty,M)}$, $ p \mapsto C((1-p)\mu + p \delta_M)$ is $L$-Lipschitz on $[p_-,\min(1,p_+)]$.
\end{lemma}

Assuming that those two lemmas are true, we obtain the result as follows. 
We set $$\hat{\nu}_n^{(\varepsilon)} = (1-p^{(\varepsilon)})\mu_n^{(\varepsilon)} + p^{(\varepsilon)}\delta_{M_{\nu}}. $$
By the triangle inequality, 
\begin{align}\label{eq:proofContinuity1}
    |C(\nu_n^{(\varepsilon)})-C(\nu^{(\varepsilon)})|\leq |C(\nu_n^{(\varepsilon)})-C(\hat{\nu}_n^{(\varepsilon)})|+|C(\hat{\nu}_n^{(\varepsilon)})-C(\nu^{(\varepsilon)})|.
\end{align}
By Lemma \ref{lemma:stepContinuity1}, the second term in \eqref{eq:proofContinuity1} converges to $0$ as $n$ tends to $+\infty$, since $(\mu_n^{(\varepsilon)})_n$ weakly converges to $\mu^{(\varepsilon)}$ when $p^{(\varepsilon)} < 1$. When $p^{(\varepsilon)} = 1$, the second term vanishes. For the first term, we apply Lemma \ref{lemma:stepContinuity2} with $p=p^{(\varepsilon)}$:  
since $(p_n^{(\varepsilon)})_n$ converges to $p^{(\varepsilon)}$, there exists $N \in \N$ such that for all $n \geq N$, $p_n^{(\varepsilon)} \in [p_- , \min(1,p_+)]$. Then, there exists $L \in \R_+$ such that for all $n \geq N $, 
$|C(\nu_n^{(\varepsilon)}) - C(\hat{\nu}_n^{(\varepsilon)})| \leq L | p_n^{(\varepsilon)} - p^{(\varepsilon)}| $ by Lemma \ref{lemma:stepContinuity2}. Therefore, the first 
term in \eqref{eq:proofContinuity1} converges to $0$, which concludes the proof of Theorem \ref{theorem:continuity}.
\end{proof} 
Let us now prove Lemmas \ref{lemma:stepContinuity1} and \ref{lemma:stepContinuity2}.

\begin{proof}[Proof of Lemma \ref{lemma:stepContinuity1}]
We consider fixed $p\in(0,1]$ and $M>0$. For $\mu$ a probability distribution on $[-\infty, M)$, we denote by $\E_{p,\mu}$ the expectation corresponding to a probability space where the weights $(X_{i,j})_{i<j}$ are i.i.d. with distribution $(1-p)\mu + p\delta_M$. By \eqref{eq:proofFormulaCnu3},
\begin{equation*}
    C((1-p)\mu + p\delta_M) = \sum_{\beta \in \mathcal{T}_m}   p^{|\beta|}(1-p)^{H(\beta)} \E_{p,\mu}\left[  \widetilde{\lambda}^{(1)}_1 \mid \alpha((X^{(2-\Tpast)}, \dots, X^{(1)})) = \beta  \right].
\end{equation*}
Notice that, conditionally to the event $(\alpha((X^{(2-\Tpast)}, \dots, X^{(1)})) = \beta)$, $(X_j^{(i)})_{i \in \llbracket 1 , |\beta| \rrbracket, j \in \llbracket 1 , \beta_i -1\rrbracket}$ are i.i.d. with distribution $\mu$. By definition of 
$\send{M}{\beta}{(x_{j}^{(i)})_{i,j}}$, 
$$ \E_{p,\mu} \left[ \widetilde{\lambda}^{(1)}_1 \mid \alpha((X^{(2-\Tpast)}, \dots, X^{(1)})) = \beta  \right] = \int \send{M}{\beta}{(x_{j}^{(i)})_{i,j}} \prod_{\substack{ 1 \leq i \leq |\beta|\\ 1 \leq j \leq \beta_i -1}} d\mu(x^{(i)}_{j}). $$
Notice that for all $\beta \in \mathcal{T}_m$ and $M>0$, $\send{M}{\beta}{\bullet}$ is continuous and bounded by $M$. 
As a consequence, $\E_{p,\mu}\left[  \widetilde{\lambda}^{(1)}_1 \mid \alpha((X^{(2-\Tpast)}, \dots, X^{(1)})) = \beta  \right]$ is a continuous function of $\mu$ for the L\'evy-Prokhorov metric. Since this conditional expectation is bounded by $M$, we obtain the continuity of $ \mu \mapsto C((1-p)\mu + p \delta_{M})$ by the dominated convergence theorem.
\end{proof}

\begin{proof}[Proof of Lemma \ref{lemma:stepContinuity2}] 
Let us consider $M>0$, $p\in (0,1]$ and $\mu \in \Mof{[-\infty,M)}$. 
Assume that $0< p_- < p < p_+$ where $p_-$ and $p_+$ will be adjusted later.
By Theorem \ref{theorem:analyticCpmu}, the map  $p \mapsto C((1-p)\mu + p \delta_M)$ is analytic on $(0,1]$.
Then, by the mean value theorem, 
\begin{equation}\label{eq:meanThm}
|C((1-p_1)\mu + p_1 \delta_M) - C((1- p_2 ) \mu + p_2 \delta_M)| 
\leq L_{\mu,p_1,p_2} |p_2 - p_1|,
\end{equation} 
for all $p_- \leq  p_1 < p_2 \leq \min(1,p_+)$, where 
$L_{\mu,p_1,p_2} := \sup\lrSet{\left| \left(\frac{\partial}{\partial p}  C((1-p)\mu + p \delta_{M})\mid_{p=p'}\right) \right|}{ p' \in (p_1,p_2)}$.
We need to dominate $L_{\mu, p_1, p_2}$ by some quantity which does not depends on $\mu$. 
For all $\beta \in \mathcal{T}_m$, $p\in(0,1]$ and $x_{j}^{(i)} \in [-\infty, M)$, we set $g(\beta, p,(x_{j}^{(i)})_{i,j} ) =p^{|\beta|}(1-p)^{H(\beta)} \send{M}{\beta}{(x_{j}^{(i)})_{i,j}} $. Then, by Proposition \ref{proposition:formuleMGS},
\begin{equation*}
    C((1-p)\mu+p\delta_M) = \sum_{\beta \in \mathcal{T}_m} \int g(\beta, p,(x_{j}^{(i)})_{i,j} ) \prod_{\substack{ 1 \leq i \leq |\beta|\\ 1 \leq j \leq \beta_i -1}} d\mu(x^{(i)}_{j}).
\end{equation*}
The function $p \mapsto g(\beta, p,(x_{j}^{(i)})_{i,j} )$ is differentiable on $(0,1]$ since it is polynomial, and we have
\begin{align*}
    \frac{\partial}{\partial p} g(\beta, p,(x_{j}^{(i)})_{i,j} ) = \left( |\beta|p^{|\beta|-1} (1-p)^{H(\beta)} - H(\beta)p^{|\beta|} (1-p)^{H(\beta)-1}\right) \cdot\send{M}{\beta}{(x_{j}^{(i)})_{i,j}} 
\end{align*}
for all $\beta \in \mathcal{T}_m \backslash\{1\}$, and $\frac{\partial}{\partial p} g(1, p,(x_{j}^{(i)})_{i,j} ) =  M .$
Since $\send{M}{\beta}{(x_{j}^{(i)})_{i,j}} \in [0,M]$, for all $\beta \in \mathcal{T}_m $, by the triangle inequality, 
\begin{align*}\left| \frac{\partial}{\partial p}g(\beta, p,(x_{j}^{(i)})_{i,j} ) \right| &\leq M \left( |\beta|p_+^{|\beta|-1} (1-p_-)^{H(\beta)} + H(\beta)p_+^{|\beta|} (1-p_-)^{H(\beta)-1}\right)  \\
&\leq M\left( \frac{|\beta|}{p_+} + \frac{H(\beta)}{1-p_-}\right)p_+^{|\beta|} (1-p_-)^{H(\beta)} \\
&\leq M\left( \frac{|\beta|}{p_+} + \frac{H(\beta)}{1-p_-}\right) p_-^{|\beta|}(1-p_-)^{H(\beta)} \cdot \left(\frac{p_+}{p_-}\right)^{|\beta|}.
\end{align*}
Since $\beta$ is triangular, $H(\beta) \leq \frac{|\beta|(|\beta|-1)}{2}$. Then, $H(\beta) \leq |\beta|^2$ and we have 
\begin{align}\label{eq:dominationProofContinuity}
\left| \frac{\partial}{\partial p}g(\beta, p,(x_{j}^{(i)})_{i,j} ) \right| &\leq M \left( \frac{1}{p_+} + \frac{1}{1-p_-}\right)p_-^{|\beta|} (1-p_-)^{H(\beta)} \cdot |\beta|^2 \left(\frac{p_+}{p_-}\right)^{|\beta|}
\end{align}
for all $\beta \in \mathcal{T}_m $. 
Since the term on the right in \eqref{eq:dominationProofContinuity} does not depend on $p$, it suffices to show that it is integrable in order to exchange the derivative and the sum by the Leibniz integral rule. With $\E_{p,\mu}$ the expectation corresponding to the probability space where the weights have distribution $(1-p)\mu + p\delta_{M}$,
 we obtain \begin{align}\label{eq:dominationProofContinuityBis}
\sum_{\beta \in \mathcal{T}_m}  p_-^{|\beta|}(1-p_-)^{H(\beta)} \cdot |\beta|^2 \left(\frac{p_+}{p_-}\right)^{|\beta|} =  \E_{p_- , \mu} \left[\Tpast^2\left(\frac{p_+}{p_-}\right)^{\Tpast}\right]
\end{align} 
by the law of total probability. 
Now, let us consider a $p_0 \in (0,p)$. 
By Theorem \ref{theorem:Zubkov}, if $ p_0 r_{p_0} > p$, the quantity in \eqref{eq:dominationProofContinuityBis} is finite with $p_- =p_0 $ and any $p_+ \in (p,  p_0r_{p_0})$. If $ p_0 r_{p_0} \leq p$, let us consider $p_-$ such that  
$p-p_0(r_{p_0}-1) < p_- <  p $. Since $r_p = \sup\Set{r \geq 1}{\E_{p,\mu}[r^{\Tpast}]< +\infty}$, $p \mapsto r_p$ is non-decreasing on $(0,1]$ since $\Tpast$ is non-increasing in $p$ by trivial coupling. Then, $0 < p-p_- <  p_0(r_{p_0}-1) \leq  p_-(r_{p_-}-1) $, which implies that $ p < p_- r_{p_-}$. Therefore, for any $p_+ \in (p,p_-r_{p_-})$, the quantity in \eqref{eq:dominationProofContinuityBis} is finite.    
We set $L =  \sum_{\beta \in \mathcal{T}_m} \left( \frac{1}{p_+} + \frac{1}{1-p_-}\right)  p_-^{|\beta|} (1-p_-)^{H(\beta)} \cdot |\beta|^2 \left(\frac{p_+}{p_-}\right)^{|\beta|}$, which is finite and depends only on $p_-$ and $p_+$. Then, by \eqref{eq:meanThm} and the Leibniz integral rule, $ p \mapsto C((1-p)\mu + p\delta_{M})$ is $L$-Lipschitz on $[p_-,\min(1,p_+)]$.
\end{proof}

\subsection{Strict monotonicity of \texorpdfstring{$C(\nu)$}{TEXT}}\label{subsec:strictMono}

In this subsection, we consider $\nu_j \in \mathcal{M}_1$ of the form $(1-p_j)\mu_j + p_j\delta_M$ for $j=1,2$, where $M>0$, $ p_1, p_2 \in (0,1]$ and $\mu_1,\mu_2 \in \Mof{[-\infty,M)}$. 
Assuming that $\nu_1$ is dominated by $\nu_2$ for the stochastic order, there exists $(X^{(n)}_{i}(\nu_j))_{i,n \geq 1}$ i.i.d. random variables with distribution $\nu_j$ such that $X^{(n)}_i(\nu_1) \leq X^{(n)}_i(\nu_2)$ almost surely for all $i,n \geq 1$ by trivial coupling. We set $\lambda$ and $\theta$ to be  the MGS with respective weights sequences $(X^{(n)}_i(\nu_1))_{i,n \geq 1}$ and $(X^{(n)}_i(\nu_2))_{i,n \geq 1}$, and starting configuration $\delta_0$. 

\begin{proposition}\label{proposition:formulaErgodicMGS}
With $\lambda$ and $\theta$ as previously defined, consider $\xi_n := \inf\{i \in \N, X^{(n)}_i(\nu_1) = M\}$ for all $n \in \N$. Let $T_1 < T_2 < \dots$ be the enumeration of the elements of  $\mathcal{R} =  \Set {n \in \N}{ \forall i \in \N,\xi_{n + i-1} \leq i}$. Then, 
\begin{align}
    C(\nu_1 ) = \gamma_{1-p_1} \E[\lambda^{(T_2-1)}_1 - \lambda^{(T_1-1)}_1], \label{align:1Inc}\\
    C(\nu_2 ) = \gamma_{1-p_1} \E[\theta^{(T_2-1)}_1 - \theta^{(T_1-1)}_1], \label{align:2Inc}
\end{align}
    where $\gamma_q := \prod_{k \geq 1} (1-q^k)$ for all $q \in [0,1)$.
\end{proposition}
\begin{proof}
For $j=1,2$, the sequences $( (X^{(l)}(\nu_j))_{T_k \leq l < T_{k+1}})_{k \in \N}$ are stationary by definition of $(T_k)_{k \in \N}$. By Lemma \ref{lemma:keyLemmaTr} and Remark \ref{remark:keyLemmaTr}, $(\lambda^{(T_{k+1}-1)}_1 - \lambda^{(T_k-1)}_1)_{k \in \N} $ depends only on $ (X^{(l)}(\nu_1))_{T_k \leq l < T_{k+1}}$ . Therefore, $(\lambda^{(T_{k+1}-1)}_1 - \lambda^{(T_k-1)}_1)_{k \in \N} $ is a stationary sequence.  
If we set $K_n = \max\Set{k \geq 0}{ T_k -1 \leq n}$, then $\lambda^{(n)}_1 = \lambda^{(T_1-1)}_1 + \sum_{k = 1}^{K_n-1} (\lambda^{(T_{k+1}-1)}_1 - \lambda^{(T_k-1)}_1 ) +(\lambda^{(n)}_1-\lambda^{(T_{K_n} -1)}_1)   $. 
By ergodicity, 
$$\frac{K_n}{n} \xrightarrow[n \rightarrow \infty]{a.s.} \Proba( 1 \in \mathcal{R})= \gamma_{1-p} > 0 $$ and \[ C( \nu_{ 1} ) = \lim_{n\rightarrow +\infty} \frac{1}{n}\sum_{k=1}^{K_n-1} \E\left[ \lambda^{(T_{k+1}-1)}_{ 1}-\lambda^{(T_{k}-1)}_{ 1} \right].  \]
Therefore, we obtain \eqref{align:1Inc}. To obtain \eqref{align:2Inc}, we notice that $X_i^{(n)}(\nu_1) = M $  implies that $X_i^{(n)}(\nu_2) = M $.  
Therefore, $(\theta^{(T_{k+1}-1)}_1 - \theta^{(T_k-1)}_1 )_{k \in \N}$ is also stationary by Lemma \ref{lemma:keyLemmaTr} and Remark \ref{remark:keyLemmaTr}, and a similar reasoning leads to \eqref{align:2Inc}.
\end{proof}

The following lemma allows us to compare the positions of the fronts in two MGS with comparable weights sequences:
\begin{lemma}\label{lemma:compareMGSs}
    Consider $(x^{(n)})_{n \in \N}$ and $(y^{(n)})_{n \in \N}$ two deterministic sequences of elements of $\mathcal{W}_{\leq M}$. Set $(\lambda^{(n)})_{n\in \Z_+}$ and  $(\theta^{(n)})_{n\in \Z_+}$ the deterministic MGS with starting configurations $ \lambda^{(0)} = \theta^{(0)} = \delta_0 $ and respective weights sequences $(x^{(n)}_i)_{i,n \in \N}$ and $(y^{(n)}_i)_{i,n \in \N}$.  More explicitly, for all $n \in \N$, set $\lambda^{(n)} = \lambda^{(n-1)} + \delta_{\mathfrak{m}(\lambda^{(n-1)},x^{(n)})}$ and $\theta^{(n)} = \theta^{(n-1)} + \delta_{\mathfrak{m}(\theta^{(n-1)},y^{(n)})}$. If  $ x^{(n)}_i \leq y^{(n)}_i$ for all $i,n\in\N$, then $ \lambda^{(n)}_i \leq \theta^{(n)}_i $ for all $n \in \N$ and $ 1 \leq i \leq n+1$.
\end{lemma}

\begin{proof}
    Let us show that $ \lambda^{(n)}_i \leq \theta^{(n)}_i $ for all $ 1 \leq i \leq n+1$ by induction on $n$. For $n = 0$, the result is clearly true. Assume that the result is true for $n-1 \in \Z_+$: $ \lambda^{(n-1)}_i \leq \theta^{(n-1)}_i $ for all $ 1 \leq i \leq n$. Therefore, by definition of $\mathfrak{m}$,  \[ \mathfrak{m}(\lambda^{(n-1)},x^{(n)}) =\max_{1 \leq i \leq n}(\lambda^{(n-1)}_i + x^{(n)}_i) \leq \max_{1 \leq i \leq n}(\theta^{(n-1)}_i + y^{(n)}_i) = \mathfrak{m}(\theta^{(n-1)},y^{(n)}) .\]
    Note that  $(\lambda^{(n)}_i)_{1\leq i \leq n+1}$ (resp. $(\theta^{(n)}_i)_{1\leq i \leq n+1}$) corresponds to the elements of $\Set{\lambda^{(n-1)}_i}{1 \leq i \leq n} \cup \{\mathfrak{m}(\lambda^{(n-1)},x^{(n)})\}$ (resp. $\Set{\theta^{(n-1)}_i}{1 \leq i \leq n} \cup \{\mathfrak{m}(\theta^{(n-1)},y^{(n)})\}$) counted with repetitions in non-increasing order.  
    
    Consider $Z$ a uniform random variable on $\llbracket 1 , n + 1 \rrbracket $ and set $S_x = \lambda^{(n-1)}_Z$ and $S_y = \theta^{(n-1)}_Z$ when $1 \leq Z \leq n$, and set $S_x = \mathfrak{m}(\lambda^{(n-1)}, x^{(n)})$ and $S_y = \mathfrak{m}(\theta^{(n-1)}, y^{(n)})$ when $Z=n+1$. By construction, $S_x \leq S_y$ almost surely. Therefore, if $F_x$ and $F_y$ are the respective cumulative distribution functions of $S_x$ and $S_y$, then  $F_x \geq F_y $. Notice that $S_x$ (resp. $S_y$) has the same distribution that $S'_{x} := \lambda^{(n)}_Z $ (resp. $S'_{y} := \theta^{(n)}_Z $). By the inequality on the cumulative distribution functions of $S'_x$ and $S'_y$, it is straightforward that $\lambda^{(n)}_i \leq \theta^{(n)}_i $ for all $1\leq i \leq n+1$.
\end{proof}

\begin{proof}[Proof of Theorem \ref{theorem:strictIncreasing}] 
Assume that $\nu_2$ strictly dominates $\nu_1$ in the sense that for all $t\in \R$, $\mu_1([-\infty , t]) \geq \mu_2([-\infty , t])$, and this inequality is strict for at least one $t \in \R$. 

By Proposition \ref{proposition:formulaErgodicMGS}, it suffices to prove that $\E[\lambda^{(T_2-1)}_1-\lambda^{(T_1-1)}_1] <  \E[\theta^{(T_2-1)}_1-\theta^{(T_1-1)}_1] $. By the law of total probabilities, 
$$\E[\lambda^{(T_2-1)}_1-\lambda^{(T_1-1)}_1] =  \sum_{\beta \in \mathcal{T}_m} p_1^{|\beta|}(1-p_1)^{H(\beta)} \E[\lambda^{(T_2-1)}_1-\lambda^{(T_1-1)}_1 | \alpha(X^{(T_1)}(\nu_1), \dots, X^{(T_2-1)}(\nu_1)) = \beta],$$ and a similar formula holds for $(\theta^{(n)}):$ $$\E[\theta^{(T_2-1)}_1-\theta^{(T_1-1)}_1] =  \sum_{\beta \in \mathcal{T}_m} p_1^{|\beta|}(1-p_1)^{H(\beta)} \E[\theta^{(T_2-1)}_1-\theta^{(T_1-1)}_1 | \alpha(X^{(T_1)}(\nu_1), \dots, X^{(T_2-1)}(\nu_1)) = \beta].$$ 

Notice that $\alpha((X^{(T_1)}(\nu_j),\dots ,X^{(T_2-1)}(\nu_j)))$ is triangular minimal for $j=1$ by definition of $T_1$ and triangular for $j=2$ since $X^{(n)}_i(\nu_1) \leq X^{(n)}_i(\nu_2)$ for all $i,n \geq 1$. Therefore, by Lemma \ref{lemma:keyLemmaTr}, $\lambda^{(T_2-1)}_1-\lambda^{(T_1-1)}_1 $ (resp. $\theta^{(T_2-1)}_1-\theta^{(T_1-1)}_1$) depends only on $\lambda^{(T_1-1)}_1$ and $(X^{(n)}(\nu_1))_{n \geq T_1}$ (resp. $\theta^{(T_1-1)}_1$ and $(X^{(n)}(\nu_2))_{n \geq T_1}$).  As a consequence, $\lambda^{(T_2-1)}_1-\lambda^{(T_1-1)}_1 \leq \theta^{(T_2-1)}_1-\theta^{(T_1-1)}_1$ almost surely by Lemma \ref{lemma:compareMGSs}. 
Therefore, it suffices to find $\beta \in \mathcal{T}_m $ such that $$ \E[\lambda^{(T_2-1)}_1-\lambda^{(T_1-1)}_1 | \alpha(X^{(T_1)}(\nu_1), \dots, X^{(T_2-1)}(\nu_1)) = \beta] < \E[\theta^{(T_2-1)}_1-\theta^{(T_1-1)}_1 | \alpha(X^{(T_1)}(\nu_1), \dots, X^{(T_2-1)}(\nu_1)) = \beta].  $$ 

For $k \in \N$, we consider the word $\beta^{(k)} = (1, \dots, 1, (k+1), k, \dots, 3, 2 )\in \mathcal{T}_m$ of size $2k$. For $k \in \N$ and $t\in\R$ such that $ M(1-k) < t \leq M(2-k) $, we also set $A_{k,t}$ to be the intersection of the events 
\begin{gather*}
    \{X^{(T_1)}_1(\nu_1) = \dots =X^{(T_1+k-1)}_1(\nu_1) = X^{(T_1+k)}_{k+1}(\nu_1) = X^{(T_1+k+1)}_{k}(\nu_1) = \dots = X^{(T_1+2k-1)}_{2}(\nu_1) =M\}, \\
    \{X^{(T_1+k)}_{1}(\nu_2) \in [t, M (2-k)]\},\\
    \{X^{(T_1+k)}_{1}(\nu_1) < X^{(T_1+k)}_{1}(\nu_2)\},\\  \bigcap_{\substack{ 1 \leq n \leq 2k \\ 1 \leq i < \beta^{(k)}_n \\ (i,n) \neq (1,k+1)}} \{ X^{(T_1 + n-1)}_i(\nu_2) \leq t \}.
\end{gather*}

Let us explain the reason why we consider this event in the case where $T_1=1$, $\lambda^{(T_1-1)}_1=0$ and $\theta^{(T_1-1)}_1=0$ (which is easy to extend to the general case). Assume that the event $A_{k,t}$ is realized.   Then for both processes $\lambda$ and $\theta$, the particle added at time $n$ for $1\leq n\leq k$ is placed at position $nM$. The particle added at time $k+1$ for $\lambda$ (resp. for $\theta$) is placed at some position $P_\lambda$ such that $M\leq P_\lambda \leq 2M$ (resp. $P_\theta$ such that $M< P_\theta \leq 2M$). Note that $P_\lambda<P_\theta$. Finally, the particle added at time $k+1+n$ for $1\leq n\leq k-1$ is placed at position $P_\lambda+nM$ for the process $\lambda$ and at position $P_\theta+nM$ for the process $\theta$. See Figure \ref{fig:8} for an illustration. 

We have $\lambda^{(T_2-1)}_1-\lambda^{(T_1-1)}_1 = P_\lambda+(k-1)M $ and $\theta^{(T_2-1)}_1-\theta^{(T_1-1)}_1 = P_\theta+(k-1)M$, hence $ \lambda^{(T_2-1)}_1-\lambda^{(T_1-1)}_1 < \theta^{(T_2-1)}_1-\theta^{(T_1-1)}_1$. To complete the proof, it suffices to find $t$ and $k$ such that $\Proba(A_{k,t})> 0$ and $ M(1-k) < t \leq M(2-k) $.

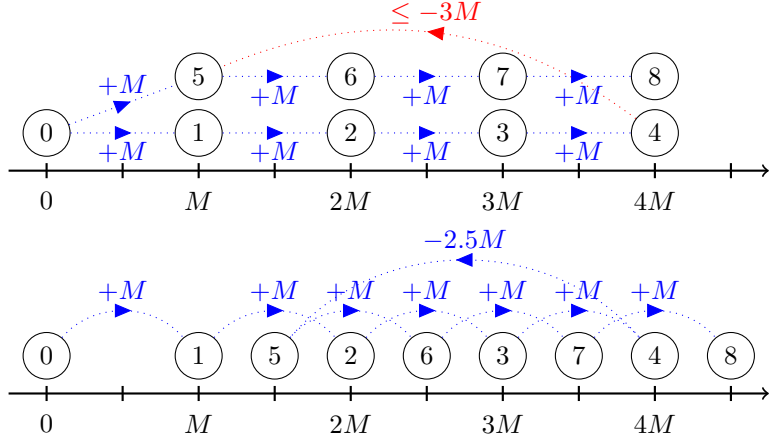
\begin{figure}[!t]
\begin{minipage}{\textwidth}
    \centering  
    \input{figures/fig8-1.tikz}
\end{minipage}
\begin{minipage}{\textwidth}
    \centering
    \input{figures/fig8-2.tikz}
\end{minipage}
    \caption{Illustration of the particles added between steps $T_1=1$ and $T_2-1 = 8$ for an MGS with weight distribution $\nu_1$ at the top and $\nu_2$ at the bottom conditionally on the event $A_{k,t}$, with $k =4 $ and  $ -3M < t \leq -2.5M $. We also assume here that $\lambda_1^{(0)}=\theta_1^{(0)}=0$. The particles are labeled by their appearance times. For this realization, $X_1^{(T_1+k)}(\nu_2) = -2.5M$ and $X_1^{(T_1+k)}(\nu_1) \leq -3M$. 
   }
    \label{fig:8}
\end{figure}

For any $t \in \R$, we set $k_t = \lceil tM^{-1} \rceil$ so that $M(k_{t}-1) < t \leq Mk_t$. 
Let $X$ and $\widetilde X$ be two random variables of respective laws $\nu_1$ and $\nu_2$ coupled by trivial coupling such that $X\leq \widetilde X$.  
For $t\in\R$, we define the events:
\begin{itemize}
    \item $B_t = \{ \widetilde X \leq t \}$,
    \item $C_t = \{X< \widetilde X \}\cap\{  t \leq \widetilde X \leq Mk_t \} $.\label{enum2inc}
\end{itemize}
We search for $t \in \R$ such that $\Proba(B_t) > 0$ and $\Proba(C_t) > 0$, this will imply that $\Proba(A_{k_t,t})>0$. Since $ \Proba(X< \widetilde X) >0 $, there exists $t \in \R$ such that $\Proba(C_t)>0$. Let us consider $k' =  \max\Set{k_t \in \Z}{\exists t \in \R, \Proba(C_t) > 0}$. Notice that $t \mapsto \Proba(B_t)$ and  $t \mapsto \Proba(C_t)$ are respectively right continuous non-decreasing and left continuous non-increasing functions on $(M(k'-1), Mk']$. 
We define  
$$t_B = \inf\Set{t \in (M(k'-1),Mk']}{\Proba(B_t)>0}$$
and 
$$t_C = \sup\Set{t \in (M(k'-1),Mk']}{\Proba(C_t)>0}.$$ 
Let us prove that $M(k'-1)  \leq  t_B \leq t_C \leq Mk'$. If $M(k'-1)  =  t_B$, these inequalities are clearly true. Assume that $M(k'-1)  <  t_B $. Notice that $\nu_2([-\infty, t_B)) = 0$ by definition of $t_B$. Therefore, $t \mapsto \Proba(C_t)$ is constant on $( M(k'-1) , t_B) $ by definition of $C_t$. Since $t \mapsto \Proba(C_t)$ is non-increasing and non-zero on $(M(k'-1),Mk']$, we have $\Proba(C_{t_B-\varepsilon})>0$ for all small enough $\varepsilon >  0$, which implies that $t_B \leq t_C$ by definition of $t_C$.

If $ t_B < t_C $, then we take $t \in (t_B, t_C)$ and we have $\Proba(B_t)>0$ and $\Proba(C_t)>0$. If $t_B = t_C$ (which implies that $M(k'-1)<t_B$), we notice that 
$$\Proba(\{X< \widetilde X \}\cap\{  t_B < \widetilde X \leq Mk' \}) =0$$
by definition of $t_C$. In addition, \[\Proba(\{X< \widetilde X \}\cap\{ M(k'-1) < \widetilde X \leq Mk' \}) >0\] by definition of $k'$. Therefore, 
$$\Proba(\{X< \widetilde X\} \cap  \{M(k'-1) < \widetilde X \leq t_B \} ) > 0$$
Since
$$\Proba(\{X< \widetilde X\} \cap  \{M(k'-1) < \widetilde X \leq t_B \} ) = \Proba(\{X< \widetilde X\} \cap\{   \widetilde X = t_B \})$$
by definition of $t_B$, we conclude that  $\Proba(B_t)>0$ and $\Proba(C_t)>0$ for $t=t_B=t_C$. 
\end{proof}

\section{The case of probability distributions supported by two non-negative real numbers}
\label{sec:atomicMeasures}

In this section, we study the case where two non-negative real numbers support $\nu$. 

In Subsection \ref{subsec:2atom-caseinvk} and Subsection \ref{subsec:2atom-FKP}, we study the case where two elements $0 <m < M$ support $\nu$. We prove Theorem \ref{theorem:rationality2atoms} in two steps. 
The first step consists in proving the theorem for $m = \frac{M}{k}$ for some $k \in \N$. 
In this case, the MGS reduces to a Markov chain on a finite state space which makes explicit computations possible.
The second step consists in extending this result to all $0<m < M$. 
To do so, we use results from \cite{https://doi.org/10.48550/arxiv.2006.01727} in Subsection \ref{subsec:2atom-FKP} to show that for $\frac{M}{k+1} < m < \frac{M}{k}$, $C((1-p)\delta_{m} + p\delta_M)$ is obtained by linear interpolation of $C((1-p)\delta_{\frac{M}{k}} + p\delta_1)$ and $C((1-p)\delta_{\frac{M}{k+1}} + p\delta_1)$.

We give formulas for $C((1-p)\delta_m + p\delta_M)$ for values of $m\in [\frac{M}{5},5M]$ at the ends of Subsections \ref{subsec:2atom-caseinvk} and \ref{subsec:2atom-FKP}. In Subsection \ref{subsec:a=0}, we consider the case where $\nu$ is supported by $0$ and $M>0$ and we give an alternative proof of Theorem \ref{theorem:fromDutta} via the MGS. 

\subsection{Rationality in \texorpdfstring{$p $}{TEXT} for measures of the form \texorpdfstring{$(1-p)\delta_{\frac{M}{k}} + p\delta_M  $}{TEXT}}\label{subsec:2atom-caseinvk}

By the rescaling property \eqref{property:rescalingProperty}, it suffices to prove Theorem \ref{theorem:rationality2atoms} for $M =1 $. 
The following proposition is equivalent to Theorem \ref{theorem:rationality2atoms} in the particular case when $M=1$ and $m = \frac{1}{k}$ for some $k \in \N$.
\begin{proposition}\label{proposition:Capos}
For all $k \in \N$, $p \mapsto C((1-p)\delta_{\frac{1}{k}} + p\delta_1)$ is a rational function on $[0,1]$.
\end{proposition}
If $k=1$, then $C((1-p)\delta_{\frac{1}{k}} + p \delta_1)=C(\delta_1) =1$ is clearly a rational function of $p$. 

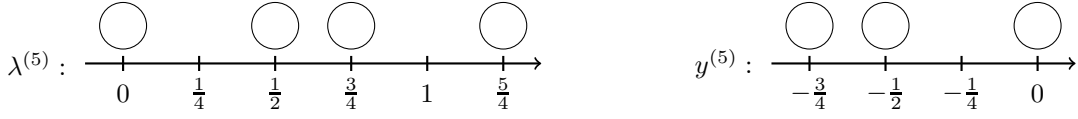
\begin{figure}[!t]
\begin{minipage}{.5\textwidth}
    \centering  
    \input{figures/fig5-1.tikz}
\end{minipage}
\begin{minipage}{.5\textwidth}
    \centering
    \input{figures/fig5-2.tikz}
\end{minipage}
\caption{Illustration of the reduction of the MGS to a Markov chain on a finite state space when $M=1$ and $k=4$. Since the dynamics of the MGS depends only on the particles at a distance less than $1$ from the front, only particles at positions $\lambda^{(n)}_1, \lambda^{(n)}_1-0.25 $, $\lambda^{(n)}_1-0.5$ and $\lambda^{(n)}_1-0.75$ matter for the evolution of the system at time $n$ by Lemma \ref{lemma:keyLemma2Atoms}.}
\label{fig:5}
\end{figure}

Now, we fix $k \geq 2$ and we set $\nu = (1-p)\delta_{\frac{1}{k}} + p \delta_1$. 
By \eqref{eq:convCouplingFK}, it suffices to study the front position of $(\lambda^{(n)})_{n \geq 0}$ an MGS with weight distribution $\nu$ and starting configuration $\lambda^{(0)} = \delta_0$ in order to study $C(\nu)$. In our case, the study of the MGS is easier than in the general case by the following lemma:
\begin{lemma}\label{lemma:keyLemma2Atoms}
We consider a starting configuration $\lambda^{(0)} = \delta_0$ and $(X^{(n)}_j)_{n,j \in \N}$ i.i.d. random variables with distribution $\nu = (1-p)\delta_{\frac{1}{k}} + p \delta_1$. We consider $(\lambda^{(n)})_{n \in \N} $ the MGS defined by $ \lambda^{(n)} = \Phi_{X^{(n)}}(\lambda^{(n-1)})$ for all $n \in \N$. Then, for all $n \in \Z_+$, if $k_n = |\Set{i \in \llbracket 1, \lambda^{(n)}(\R)\rrbracket}{ \lambda^{(n)}_i > \lambda_1^{(n)} - 1}|$ is the number of particles at distance less than $1$ from the front at time $n$, 
\begin{equation}\label{eq:keyLemma2Atoms}
    \mathfrak{m}(\lambda^{(n)}, X^{(n+1)}) = \max_{i \in \llbracket 1 , k_{n}\rrbracket}(\lambda_{i}^{(n)} + X_i^{(n+1)}).
\end{equation}

As a consequence, almost surely, for all $n \in \N$, 
\begin{enumerate}[label=\emph{B.\arabic*}]
    \item \label{enum:Ck1} The dynamics of the MGS at each step depends only on the particles at a distance less than $1$ from the front. 
    \item  \label{enum:Ck2} The position $ \mathfrak{m}(\lambda^{(n-1)}, X^{(n)})$ of the new particle at time $n$ verifies: $$ \mathfrak{m}(\lambda^{(n-1)}, X^{(n)}) - \lambda^{(n-1)}_1 \in \lrSet{\frac{i}{k}}{i \in \llbracket 1 , k \rrbracket}. $$
    \item \label{enum:Ck3} The front moves at each step:  $ \lambda^{(n)}_1 = \mathfrak{m}(\lambda^{(n-1)}, X^{(n)})>\lambda^{(n-1)}_1$. 
    \item \label{enum:Ck4} All particles have positions in $\frac{1}{k}\Z$: $\lambda_j^{(n-1)} \in \frac{1}{k}\Z$ for all $n \in \N$ and $j \in \N$. 
    \item \label{enum:Ck5} There is at most one particle at a given position. 
\end{enumerate}
\end{lemma}
\begin{proof}
We notice that $X_1^{(n+1)} \geq \frac{1}{k}$ for all $n \in \Z_+$. Therefore, if $\lambda_i^{(n)} \leq \lambda_1^{(n)} - 1$, then $ \lambda_i^{(n)} + X^{(n+1)}_i \leq \lambda_1^{(n)} + X^{(n+1)}_1 $. Therefore, by definition of $\mathfrak{m}$, we obtain \eqref{eq:keyLemma2Atoms}: $$ \mathfrak{m}(\lambda^{(n)}, X^{(n+1)}) = \max_{i \in \llbracket 1 , k_{n}\rrbracket}(\lambda_{i}^{(n)} + X_i^{(n+1)}).$$

Points \ref{enum:Ck1} to \ref{enum:Ck5} are simple consequences of \eqref{eq:keyLemma2Atoms} and of the fact that  $X_i^{(j)} \in \{ \frac{1}{k},1\}$ almost surely.
\end{proof}

For a configuration $\lambda \in \mathcal{N}$, we set $u(\lambda)$ the configuration obtained by removing from $\lambda$ all the particles at positions smaller or equal to $\lambda_1-1$  and by shifting all the particles so that the front position in $u(\lambda)$ is $0$. More explicitly, for all $\lambda \in \mathcal{N}$, $$ u(\lambda) = \sum_{i \in \N} \delta_{\lambda_i - \lambda_1} \mathds{1}_{\lambda_1 - \lambda_i < 1} .$$
See Figure \ref{fig:5} for an illustration.
By Lemma \ref{lemma:keyLemma2Atoms}, we can now define our Markov chain $(y^{(n)})_{n \in \Z_+}$:
\begin{proposition}\label{proposition:4.3}
Let us consider $(\lambda^{(n)})_{n \in \Z_+}$ an MGS with starting configuration $\lambda^{(0)} = \delta_0$ and measure $\nu = (1-p)\delta_{\frac{1}{k}} + p \delta_1$. If $y^{(n)} := u(\lambda^{(n)})$ for all $n \in \Z_+$, then $(y^{(n)})_{n \in \Z_+}$ is a Markov chain on the finite state space $E_k = \Set{ \delta_0 + \sum_{i = 1}^{k-1} x_i\delta_{\frac{-i}{k}}}{ (x_1, \dots, x_{k-1})\in \{0,1\}^{k-1}}$. In addition, there is a unique recurrent class for this Markov chain, and for all $n \in \N$, \begin{equation}\label{eq:propositionMC}
    \lambda^{(n)}_1 = \sum_{i = 1}^n -y^{(i)}_2.
\end{equation}   
\end{proposition}
\begin{proof}
It is easy to check that $ u(\Phi_{X^{(n)}}(\lambda^{(n-1)})) = u(\Phi_{X^{(n)}}(u(\lambda^{(n-1)})))$ for all $n \in \N$ since the dynamics of the MGS depends only on the particles at a distance less than $1$ from the particle(s) at the front by \eqref{eq:keyLemma2Atoms}. Therefore, $ u(\lambda^{(n)}) = u(\Phi_{X^{(n)}}(u(\lambda^{(n-1)}))) $  for all $n \in \N$. 
Then, $$ y^{(n)} = u \circ \Phi_{X^{(n)}}(y^{(n-1)} ) $$ for all $n \in \N$ by definition of $(y^{(n)})_{n \in \Z_+}$. Since $(X^{(n)})_{n \in \N}$ are i.i.d., $(y^{(n)})_{n \in \Z_+}$ is a Markov chain.

By \ref{enum:Ck2}, \ref{enum:Ck4}, \ref{enum:Ck5} from Lemma \ref{lemma:keyLemma2Atoms} and by construction of $(y^{(n)})_{n \in \Z_+}$,  $y^{(n)} \in E_k$ for all $n \in \Z_+$. 

We prove the uniqueness of the recurrent class by proving that, starting from any configuration in $E_k$ for the Markov chain, there is a non-zero probability for the next configuration to be $\delta_0$. For all $\lambda_0 \in E_k$, it is easy to compute that $\Proba(y^{(n)} = \delta_0 \mid y^{(n-1)} = \lambda_0) = \Proba( X^{(n)}_1 = 1 ) = p > 0 $ for all $n \in \N$.

It remains to prove \eqref{eq:propositionMC}. By \ref{enum:Ck3}, $\lambda^{(n)}_1 - \lambda^{(n-1)}_1 = \lambda^{(n)}_1 - \lambda^{(n)}_2$ for all $n \in \N$. As a consequence, $\lambda^{(n)}_1 = \sum_{i=1}^{n} (\lambda^{(i)}_1 - \lambda^{(i)}_2). $
Therefore, by construction of $(y^{(n)})_{n \in \Z_+}$, $\lambda^{(n)}_1 = \sum_{i=1}^{n} (y^{(i)}_1 -y^{(i)}_2). $
Since $y^{(i)}_1 =0$ for all $i \in \N$, we obtain \eqref{eq:propositionMC}.
\end{proof}

\begin{proof}[Proof of Proposition \ref{proposition:Capos}]
We can now complete the proof of Proposition \ref{proposition:Capos}. By Lemma \ref{lemma:keyLemma2Atoms}, there exists a unique stationary distribution for the Markov chain $(y^{(n)})$. Let us denote by $\mu_{p,k}$ this stationary distribution.  Then, by Proposition \ref{proposition:4.3} and \eqref{eq:convCouplingFK}, $$ \frac{1}{n}\sum_{i = 1}^n -y^{(i)}_2  \xrightarrow[n \rightarrow \infty]{a.s.} C(\nu)$$

By the ergodic theorem for Markov chains, $C(\nu) = \E[- \widetilde{y}_2]$ where $\widetilde{y}$ has distribution $\mu_{p,k}$. 
Since $(X_i^{(n)})_{i,n \in \N}$ are i.i.d. with distribution $\nu = (1-p)\delta_{\frac{1}{k}} + p \delta_1$, the transition probabilities for this Markov chain are polynomials in $p$. As a consequence, since one can obtain $\mu_{p,k}$ as an eigenvector for the eigenvalue $1$ of the transition matrix of $(y^{(n)})_{n \geq 1}$, its coefficients are rational functions in $p$. Therefore, the theorem is proved.
\end{proof}

Since $(y^{(n)})_{n\geq 1}$ is a Markov chain on a finite state space, one can explicitly compute the stationary distribution for this model by computating the eigenspace of the transition matrix for the eigenvalue $1$. As a consequence, we can give an explicit formula for $C((1-p)\frac{1}{k} + p\delta_1)$ for small values of $k$: 
 \begin{align*}
    &C(\delta_1) = 1 ,\\
    &C((1-q) \delta_1 +q\delta_{\frac{1}{2}}) = \frac{1}{2}(2-q) ,\\
    &C((1-q) \delta_1 +q\delta_{\frac{1}{3}}) = \frac{3+q-4q^2+q^3}{3(1+q-q^2)},\\
    &C((1-q) \delta_1 +q\delta_{\frac{1}{4}}) = \frac{4-3q+2q^2-q^4+4q^5-11q^6+7q^7-q^8}{4(1+q^3-q^4+2q^5-3q^6+q^7)},\\
&C((1-q) \delta_1 +q\delta_{\frac{1}{5}})= \frac{1}{5}(5+q-6q^2+11q^3-8q^4-7q^5+q^6+14q^7-23q^8+16q^9+23q^{10} -68q^{11}\\&+81q^{12} -68q^{13}+33q^{14}+28q^{15}-85q^{16}+90q^{17}-48q^{18}+12q^{19}-q^{20}) \times (1+q-q^2+2q^3-3q^5+q^6\\&+q^7-2q^8-q^9+10q^{10}-16q^{11}+12q^{12}-2q^{13}-11q^{14}+23q^{15}-27q^{16}+19q^{17}-7q^{18}+q^{19})^{-1}.
\end{align*}

\subsection{Rationality in \texorpdfstring{$p $}{TEXT} for measures of the form \texorpdfstring{$(1-p)\delta_{m} + p\delta_M  $}{TEXT} for \texorpdfstring{$0<m<M$}{TEXT}}
\label{subsec:2atom-FKP}

In this subsection, we complete the proof of Theorem \ref{theorem:rationality2atoms}. 
To do so, we recall the notion of skeleton points from 
\cite{https://doi.org/10.48550/arxiv.2006.01727}, which are the renovation events used by Foss and Konstantopoulos to prove the convexity of $m \mapsto C((1-p)\delta_m + p\delta_1)$ and to characterize its points of non-differentiability for all $p \in (0,1)$. 

We consider $(X_{i,j})_{i<j \in \Z}$ i.i.d. random variables with distribution $\nu=(1-p)\delta_m + p\delta_1$ with $m \leq 1$. An integer $ n \in \Z$ is a \emph{skeleton point on the left} if for every integer $k<n$, there exists a directed path from $k$ to $n$ made of edges with weight $1$. Similarly, $ n \in \Z$ is a \emph{skeleton point on the right} if for every integer $k>n$, there exists a directed path from $n$ to $k$ made of edges with weight $1$. A \emph{skeleton point} is a skeleton point on the left and the right.
In other words, $n \in \Z$ is a skeleton point if, for all $ a < n < b$,  there is a directed path of edges with weights $1$ from $a$ to $n$ and a directed path of edges with weights $1$ from $n$ to $b$. Analogously, one could define the notion of skeleton points in other weighted directed acyclic graphs. We give an example of skeleton points for a complete graph with $6$ vertices in Figure \ref{fig:6}.

\begin{remark}
    Notice that the renovation events that we use in Subsection \ref{subsec:statMGS}, which consist in triangular words, 
    precisely correspond to the skeleton points on the right when $\nu= (1-p)\delta_m + p\delta_1$ by Remark \ref{remark:interpretTr}.   
\end{remark}

\begin{figure}
    \centering
    \input{figures/fig6.tikz}
    \caption{Illustration of skeleton points in a graph with $6$ vertices. The red edges have weights $m$, the blue ones have weights $1$. In this graph, $2$ is a skeleton point on the right, $4$ is a skeleton point on the left, and $3$ is a skeleton point.}
    \label{fig:6}
\end{figure}
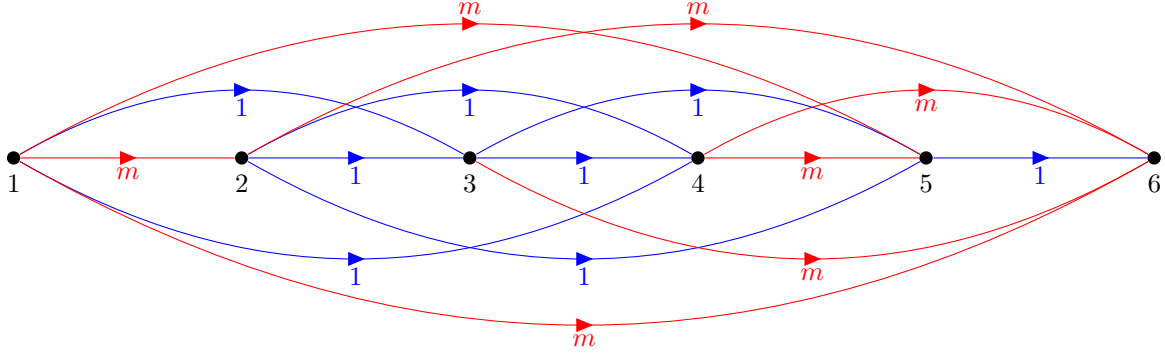

For $n_1 < n_2 \in \Z$, let $W_{n_1,n_2}$ be the weight of a heaviest path between the vertices $n_1$ and $n_2$ as defined in \eqref{eq:defWn2}. 
We also denote by $(\Gamma_i)_{i\in\Z}$ the sequence of all the skeleton points in $\Z$ with the convention that $\dots < \Gamma_{-1} < \Gamma_0 \leq 0  < \Gamma_1 < \dots  $. To prove Theorem \ref{theorem:rationality2atoms} for all $m \in (0,1]$, we use the following result from \cite{https://doi.org/10.48550/arxiv.2006.01727}: 
\begin{proposition}[{\cite[Proposition 4]{https://doi.org/10.48550/arxiv.2006.01727}}]\label{proposition:FKP4}
For all $0 < m \leq 1$ and $p\in]0,1]$, 
\begin{align*}
    C((1-p)\delta_m + p\delta_1) = \gamma_{1-p}^2 \E\left[ W_{\Gamma_1,\Gamma_2} \right],
\end{align*} where $\gamma_q = \prod_{n=1}^{\infty}(1-q^n)$ for all $q \in [0,1[$.  
\end{proposition}

By Proposition \ref{proposition:FKP4}, we only need to study the expectation of the heaviest path between two consecutive skeleton points.  For $n\in\N$, let $\SP_n$ be the set of weights $(x_{i,j})_{1 \leq i < j \leq n} \in \{m,1\}^{\binom{n}{2}}$ such that $1$ and $n$ are the only two skeleton points for those edge weights in a complete graph with $n$ vertices. Let us consider $\SP := \cup_{n \geq 2}\SP_n $. For $x = (x_{i,j})_{1 \leq i < j \leq n} \in \SP$ and $\pi=(i_1, \dots , i_k)$ a directed path from $i_1 = 1$ to $i_k = n$,  
denote by $N_x(\pi)= |\Set{j \in \llbracket 1, k-1 \rrbracket}{x_{i_j,i_{j+1}}=1}|$ and $\overline{N}_x(\pi)= |\Set{j \in \llbracket 1, k-1 \rrbracket}{x_{i_j,i_{j+1}}=m} |$ 
the respective numbers of edges with weights $1$ and $m$ in the path $\pi$, and denote by $w_x(\pi) = \sum_{c=1}^{k-1} x_{i_c,i_{c+1}}$  the weight of the path $\pi$. 
For $\pi_{\max}$ a directed path from $1$ to $n$ maximising $w_x$, we set $N_x = \overline{N}_x(\pi_{\max})$ and $N_x = \overline{N}_x(\pi_{\max})$. With those notations, $W_{\Gamma_1, \Gamma_2} = N_{X'} + m \overline{N}_{X'}$, where $X' = (X_{i,j})_{\Gamma_1 \leq i < j \leq \Gamma_2}$.

In general, $N_{X'}$ and $\overline{N}_{X'}$ are not uniquely defined since there could be several ways to obtain the maximal weight in terms of the number of edges with weight $1$ and $m$ in a maximal path for fixed weights. Yet, it is possible to show that those two quantities are uniquely defined for all $0< m \leq 1$ such that $m$ is not the reciprocal of an integer greater or equal to $2$. 
If $\Proba_{p,m}( N  \text{ and } \overline{N} \text{ are not uniquely defined})>0$ for some $p\in (0,1)$, then $m$ is called a \emph{critical point}. It has been proved in \cite{https://doi.org/10.48550/arxiv.2006.01727} that the positive critical points are all integers greater than $1$ and their reciprocals: 
\begin{proposition}[{\cite[Theorem 5]{https://doi.org/10.48550/arxiv.2006.01727}}]\label{proposition:critPoint}
For $0<m\leq 1$, $m$ is critical if and only if $m = \frac{1}{k}$ for some integer $k \geq 2$. 
\end{proposition}

Since the probability that $X' = x$ is positive for all $x \in \SP$ and $p\in (0,1)$, there is a deterministic definition of the critical points: 
$m$ is a critical point if and only if there exists $x \in \SP$ and $\pi_1 =(i_1, \dots, i_d)$ and $\pi_2=(j_1, \dots, j_r)$ two directed paths with $i_1 = j_1 = 1$ and $i_d = j_r = n$ such that $w_x(\pi_1) = w_x(\pi_2) = \max_{\pi} w_x(\pi)$ and $ N_x(\pi_1) \neq N_x(\pi_2).$
With our notations, notice that $\max_{\pi}w_x(\pi)  = N_x + m \overline{N}_x.$

In Subsection \ref{subsec:2atom-caseinvk}, we proved that $p \mapsto C((1-p)\delta_m + p \delta_1)$ is a rational function for any critical point $0<m<1$. 
The following lemma enables us to extend this result to all $m\in (0,1)$. 

\begin{lemma}\label{coro:FKP-NNbar}
For any $k \in \N$ and $x \in \SP$, $N_x$ and $\overline{N}_x$ are constant functions of $m$ on $(\frac{1}{k+1},\frac{1}{k})$. 
\end{lemma}

\begin{proof}
Consider an integer $k\geq 1$ and $x=x(m)\in\SP_n$ for some $n \in \N$. By Proposition \ref{proposition:critPoint}), every element $m \in (\frac{1}{k+1}, \frac{1}{k})$ is non-critical. Therefore, $N_{x(m)}$ and $\overline{N}_{x(m)}$ are well-defined for all $m \in (\frac{1}{k+1}, \frac{1}{k})$.
Let us fix $c$ an element of $(\frac{1}{k+1},\frac{1}{k})$ (for instance, $c:=(\frac{1}{k+1}+\frac{1}{k})/2$) and set $$U_{x,c}=\lrSet{ m \in \left(\frac{1}{k+1},\frac{1}{k}\right)}{ N_{x(m)}=N_{x(c)} \text{ and } \overline{N}_{x(m)}=\overline{N}_{x(c)}}.$$ 
To prove Lemma \ref{coro:FKP-NNbar}, let us show that $U_{x,c} = (\frac{1}{k+1},\frac{1}{k})$. We show that $U_{x,c}$ is non-empty, closed and open relatively to $(\frac{1}{k+1},\frac{1}{k})$. The first point is clear since $c \in U_{x,c}$. 

Firstly, let us show that $U_{x,c}$ is closed relatively to $(\frac{1}{k+1},\frac{1}{k})$.
If $(m_l)_{l \geq 1} \in (U_{x,c})^{\N}$ converges to some $m \in (\frac{1}{k+1},\frac{1}{k})$ as $l$ tends to infinity, then $\max_{\pi}w_{x(m_l)}(\pi)$ converges to $\max_{\pi}w_{x(m)}(\pi)$ by continuity of $m \mapsto \max_{\pi}w_{x(m)}(\pi)$, where we take the $\max$ over all directed paths starting at $1$ and ending at $n$. Since $m_l \in U_{x,c}$ for all $l \geq 1$, $\max_{\pi}w_{x(m_l)}(\pi) = N_{x(c)} + m_l \overline{N}_{x(c)}$, which converges to $ N_{x(c)} + m \overline{N}_{x(c)}$ as $l$ tends to infinity. Therefore, $\max_{\pi}w_{x(m)}(\pi) = N_{x(c)} + m \overline{N}_{x(c)}$. By the non-criticality of $m$ and the definition of $N$ and $\overline{N}$, we have proved that $m \in U_{x,c}$, since there is a path with $N_{x(c)}$ edges with weight $1$ and $\overline{N}_{x(c)}$ edges with weight $m$ in the sequence $x$. 
Then, $U_{x,c}$ is closed relatively to $(\frac{1}{k+1},\frac{1}{k})$. 

Now, we prove  that $U_{x,c}$ is open relatively to $(\frac{1}{k+1},\frac{1}{k})$. By contradiction, let us assume there exists $m \in U_{x,c}$ and $(m_l)_{n \geq 1} \in (\frac{1}{k+1},\frac{1}{k})^{\N}$ such that $m_l$ converges to $m$ as $n$ tends to infinity and such that $m_l \notin U_{x,c}$ for all $n\geq 1$. Since $N_{x(m_l)}$ and $\overline{N}_{x(m_l)}$ 
are bounded by the number of vertices in the graph with weights $x$, there is an extraction $(m_{\varphi(l)})_l$ of $(m_l)_l$  such that $(N_{x(m_{\varphi(l)})})$ and $(\overline{N}_{x(m_{\varphi(l)})})$ both converge. Let us denote their respective limits by $n_0$ and $\overline{n}_0$. Since those sequences are integer-valued, they are equal to respectively $n_0$ and $\overline{n}_0$ for $n$ large enough. Then, $ \max_{\pi} w_{x(m)} (\pi) = n_0 + m \overline{n}_0 $ by continuity of $m \mapsto \max_{\pi}w_{x(m)}(\pi)$. By definition of $N$ and $\overline{N}$, and since $m$ is non-critical, we have $N_{x(m)}=n_0$ and $\overline{N}_{x(m)}=\overline{n}_0$, which contradicts the fact that $m_l \notin U_{x,c}$. As a consequence,  $U_{x,c}$ is open relatively to $(\frac{1}{k+1},\frac{1}{k})$. 

Since $U_{x,c}$ is open and closed relatively to $(\frac{1}{k+1},\frac{1}{k})$ which is connected, $U_{x,c}=(\frac{1}{k+1},\frac{1}{k})$. Equivalently, $m \mapsto N_{x,m}$ and $m \mapsto \overline{N}_{x,m}$ are constant functions on $(\frac{1}{k+1},\frac{1}{k})$.
\end{proof}

\begin{proof}[Proof of Theorem \ref{theorem:rationality2atoms}]
We consider an integer $k \geq 1$ and $p\in (0,1)$. 
By Lemma \ref{coro:FKP-NNbar}, if $\Gamma_1 < \Gamma_2$ are the first two positive skeleton points, then $ \E_{p,m}\left[N\right] $ and $\E_{p,m}\left[\overline{N}\right]$ are constant functions of $m$ on $(\frac{1}{k+1},\frac{1}{k})$ for fixed $p$. Therefore, by Proposition \ref{proposition:FKP4}, if we set 
$f_{\frac{1}{k}}(p) = \gamma_{1-p}^{2} \E_{p,c}\left[N\right]$ and $g_\frac{1}{k}(p) = \gamma_{1-p}^{2} \E_{p,c}\left[\overline{N}\right]$ where $c = \frac{1}{2}\left(\frac{1}{k}+\frac{1}{k+1}\right)$, then \begin{align}\label{formulaToExtend}
    C((1-p)\delta_m + p\delta_1) = f_{\frac{1}{k}}(p) +  m \cdot g_{\frac{1}{k}}(p)
\end{align} for all $m \in (\frac{1}{k+1}, \frac{1}{k})$ and $p \in (0,1)$. Since $m \mapsto C((1-p)\delta_m + p\delta_1)$ is continuous on $\R$ according to Theorem \ref{theorem:continuity}, \eqref{formulaToExtend} is also true for $m \in [\frac{1}{k+1}, \frac{1}{k}]$ and $p \in [0,1]$. Therefore, considering \eqref{formulaToExtend} with $m = \frac{1}{k} $ and then with $m=\frac{1}{k+1}$, we get 
\begin{align*}
\left\{
    \begin{array}{lllll}
        C((1-p)\frac{1}{k} + p\delta_1) &= &f_{\frac{1}{k}}(p) &+ &  {\frac{1}{k}} \cdot g_{\frac{1}{k}}(p) \\
        C((1-p)\frac{1}{k+1} + p\delta_1) &= &f_{\frac{1}{k}}(p) &+ &{\frac{1}{k+1}} \cdot g_{\frac{1}{k}}(p),
    \end{array}
\right.
\end{align*}
which is equivalent to \begin{align}\label{eqCaPos}
\left\{
    \begin{array}{lll}
        f_{\frac{1}{k}}(p) &= & (k+1) \cdot C((1-p)\frac{1}{k+1} + p\delta_1) - k \cdot C((1-p)\frac{1}{k} + p\delta_1) \\
        g_{\frac{1}{k}}(p) &= & k(k+1)(C((1-p)\frac{1}{k} + p\delta_1)-C((1-p)\frac{1}{k+1} + p\delta_1)).
    \end{array}
\right.
\end{align}
By Proposition \ref{proposition:Capos}, $p \mapsto C((1-p)\delta_m + p \delta_1)$ is a rational function for all $m = \frac{1}{k}$ with $k \geq 1$. Then, $f_{\frac{1}{k}}$ and $g_{\frac{1}{k}}$ are also rational functions on $(0,1)$ by \eqref{eqCaPos}, which implies that $p\mapsto C((1-p)\delta_m + p \delta_1)$ is rational on $(0,1)$. By Theorem \ref{theorem:analyticCpmu}, this function is continuous at $p=1$. By the rescaling property and Theorem \ref{theorem:analyticCpmu}, it is also continuous at $p=0$. Therefore, $p\mapsto C((1-p)\delta_m + p \delta_1)$ is rational on $[0,1]$.  
\end{proof}

Since $C((1-p)\delta_{\frac{1}{k}} + p \delta_1)$ has been computed for $k \in \llbracket 1 , 5\rrbracket$ (c.f. Subsection \ref{subsec:2atom-caseinvk}),  we obtain the following formulas for $f_{\frac{1}{k}}$ and $g_{\frac{1}{k}}$ by \eqref{eqCaPos}: 
\begin{align*}
     f_1(1-q) &= 1-q \text{, }\quad &g_1(1-q) &= q,\\
     f_{\frac{1}{2}}(1-q) &= \frac{(1-q)(1+q)}{1+q-q^2}  \text{, }\quad &g_{\frac{1}{2}}(1-q) &=  \frac{q(1-q+q^2)}{1+q-q^2}, \\
     f_{\frac{1}{3}}(1-q) &= \frac{(1-q)(1+q+q^3+2q^5-2q^6-q^7+q^8)}{(1+q-q^2)(1+q^3-q^4+2q^5-3q^6+q^7)} \text{, } \quad &g_{\frac{1}{3}}(1-q) &= \frac{q(1-q+q^2)(1+q^3-4q^4+5q^5-3q^6+q^7)}{(1+q-q^2)(1+q^3-q^4+2q^5-3q^6+q^7)}.
\end{align*}
One could also compute $f_{\frac{1}{4}}(p)$ and $g_{\frac{1}{4}}(p)$ using \eqref{eqCaPos} and the formulas we give for $C((1-p)\delta_{\frac{1}{4}} + p \delta_1)$ and $C((1-p)\delta_{\frac{1}{5}} + p \delta_1)$. As a consequence, we get an explicit formula for $p \mapsto C((1-p)\delta_m + p \delta_1)$, for $m \in [\frac{1}{k+1},\frac{1}{k} ]$ by \eqref{formulaToExtend}.

\begin{remark}\label{remark:Ck}
By the rescaling property \eqref{property:rescalingProperty} and \eqref{formulaToExtend}, $C((1-p)\delta_m + p\delta_1) = g_{\frac{1}{k}}(p) +  m \cdot f_{\frac{1}{k}}(p)$ for all $k \geq 2$ and $m \in [k , k+1]$. 
\end{remark}

\subsection{A formula for the time constant for measures of the form \texorpdfstring{$\nu = (1-p)\delta_0 + p\delta_M $}{TEXT}}
\label{subsec:a=0}

One can give an explicit formula for $p \mapsto C((1-p)\delta_{0} + p \delta_1)$. Dutta first found such a formula in \cite{dutta2020limit}. We give here an alternative proof of Theorem \ref{theorem:fromDutta} via the coupling with the MGS with weights $0$ and $1$. 
\begin{figure}[!t]
\centering
\minipage{.33\textwidth}
    \centering
    \input{figures/fig7-1.tikz}
\endminipage
\minipage{.33\textwidth}
    \centering
    \input{figures/fig7-2.tikz}
\endminipage
\minipage{.33\textwidth}
    \centering
    \input{figures/fig7-3.tikz}
\endminipage
\caption{Illustration of the case $\nu = (1-p)\delta_0 + p\delta_1$. Here, $Y_4 = 3$, $Y_5=1$ and $Y_6 = 2$ with notations from the proof of Theorem \ref{theorem:fromDutta}. Knowing $\lambda^{(4)}$, the probability to obtain $\lambda^{(5)}$ as above is $1-q^3$ since there are $3$ front particles in $\lambda^{(4)}$.}
\label{fig:7}
\end{figure}
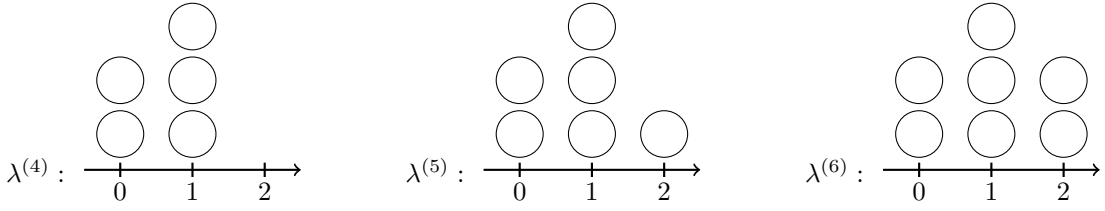

\begin{proof}[Proof of Theorem \ref{theorem:fromDutta}]
By the rescaling property \eqref{property:rescalingProperty}, it suffices to prove the result for $M=1$.

We consider an MGS $(\lambda^{(n)})_{n \geq 0}$ with weight sequences $(X^{(n)})_{n \geq 0}$ with distribution $\nu = (1-p)\delta_{0} + p \delta_1$ and starting configuration $\lambda^{(0)}=\delta_0$. Since $X^{(n)}_i \in \{0,1\}$ almost surely for all $n,i \in \N$, it is easy to show that $\lambda_i^{(n)} \in \Z_{\geq 0}$ almost surely and that \eqref{eq:keyLemma2Atoms} still holds in this case. 
Therefore, $$\mathfrak{m}(\lambda^{(n-1)}, X^{(n)}) = \lambda^{(n-1)}_1 + \max_{i \in \llbracket 1 , Y_{n-1} \rrbracket}  X^{(n)}_i$$ where $Y_k = |\Set{i \in \llbracket 1, \lambda^{(k)}(\R)\rrbracket}{ \lambda^{(k)}_i = \lambda_1^{(k)}}|$ for all $k \in \Z_+$ (see Figure \ref{fig:7} for an illustration).
Since $X^{(n)}_i \in \{0,1\}$ almost surely, $\max_{i \in \llbracket 1 , Y_{n-1} \rrbracket}  X^{(n)}_i \in \{0,1\}$ and $\mathfrak{m}(\lambda^{(n-1)}, X^{(n)}) = \lambda_1^{(n)}$ almost surely for all $n \in\N$. Notice that $Y_n=1$ if and only if the front position moved from time $n-1$ to time $n$. Then, for all $n \in \N$, we obtain 
\begin{equation}\label{eq:Dutta2}
    \lambda^{(n)}_1-\lambda^{(n-1)}_1 = \mathds{1}_{Y_n = 1}.
\end{equation}

We consider $\mathcal{F}_n$ the $\sigma$-algebra generated by $(X^{(k)})_{k \leq n}$ for all $n \in \N$. Since $Y_{n-1}$ is $\mathcal{F}_{n-1}$-measurable, $Y_{n-1}$ and $X^{(n)}$ are independent by construction. 
As a consequence, it is easy to show that $(Y_n)_{n\geq 1}$ is a discrete-time Markov chain on $\N$ with transitions $(p_{i,j})_{i,j\in \N}$ where $ p_{i,i+1}= (1-p)^i $ and $p_{i,1}=1-(1-p)^i$. 
The result of Theorem \ref{theorem:fromDutta} is trivial for $p=1$. For $p\in ]0,1[$, this Markov chain is clearly irreducible. Furthermore, a distribution $\pi$ on $\N$ is stationary if and only if
\begin{align*}
\left\{
    \begin{array}{ll}
        \forall j \geq 2\; , \pi_j = \pi_{j-1} (1-p)^{j-1}, \\
        \pi_1 = \sum_{j\geq 1} \pi_j (1-(1-p)^j).
    \end{array}
\right. 
\end{align*}

As a consequence, there is a unique stationary probability distribution $\pi$ defined by $$ \pi_j=\frac{(1-p)^{\binom{j}{2}}}{\sum_{n=1}^{\infty} (1-p)^{\binom{n}{2}}} $$ for all $j \in \N$. 
By \eqref{eq:Dutta2}, $\lambda^{(n)} = \sum_{i=1}^n \mathds{1}_{Y_i = 1}$. Then, by the ergodic theorem for Markov chains and \eqref{eq:convCouplingFK}, $C((1-p)\delta_{0} + p \delta_1) = \pi_1$, which completes the proof.
\end{proof}

\section*{Acknowledgements}
I warmly thank Sanjay Ramassamy and Arvind Singh for introducing me to the subject, for their constant support and guidance, and for their comments on the previous versions of this paper.
 
\bibliographystyle{plain}
\bibliography{refs}

\end{document}

%% file: figures/fig1.tikz
\begin{tikzpicture}[
    middlearrow/.style 2 args={
        decoration={             
            markings, 
            mark=at position 0.5 with {\arrow[xshift=3.333pt]{triangle 45}, \node[#1] {#2};}
        },
        postaction={decorate}
    },
    scale=1.5
]

    \node[circle,fill=black,inner sep=0pt,minimum size=5pt,label=below:{$0$}] (0) at (6,0) {};
    \node[circle,fill=black,inner sep=0pt,minimum size=5pt,label=below:{$1$}] (1) at (8,0) {};
    \node[circle,fill=black,inner sep=0pt,minimum size=5pt,label=below:{$2$}] (2) at (10,0) {};
    \node[circle,fill=black,inner sep=0pt,minimum size=5pt,label=below:{$3$}] (3) at (12,0) {};
    \node[circle,fill=black,inner sep=0pt,minimum size=5pt,label=below:{$4$}] (4) at (14,0) {};

    \draw[black!100,middlearrow={below}{1}] (0) -- (1);
    \draw[black!100,middlearrow={below}{0}] (1) -- (2);
    \draw[black!100,middlearrow={below}{2}] (2) -- (3);
    \draw[black!100,middlearrow={below}{-3}] (3) -- (4);

    \draw[black!100] (0) edge[bend left, middlearrow={below}{0.7}] (2);
    \draw[black!100] (1) edge[bend left, middlearrow={below}{0.5}] (3);
    \draw[black!100] (2) edge[bend left, middlearrow={below}{0.7}] (4);
    
    \draw[black!100] (0) edge[bend right, middlearrow={below}{0.9}] (3);
    \draw[black!100] (1) edge[bend right, middlearrow={below}{0.2}] (4);

    \draw[black!100] (0) edge[bend left, middlearrow={below}{1}] (4);

\end{tikzpicture}

%% file: figures/fig2-1.tikz
\begin{tikzpicture}
  [
  middlearrow/.style 2 args={
        decoration={             
            markings, 
            mark=at position 0.5 with {\arrow[xshift=3.333pt]{triangle 45}, \node[#1] {#2};}
        },
        postaction={decorate}
    },
    particle/.style={circle, draw},
    minimum size=5pt
    ]

    \node[label=left:{$\lambda :$}] at (-2,-0.5) {};
    \draw[->,thick] (-2,-0.5) -- (2,-0.5);
    
    \draw[-,thick] (-1,-0.4) -- (-1,-0.6);
    \node (labelm1) at (-1,-0.8) {$-1$};
    
    \draw[-,thick] (0.3,-0.4) -- (0.3,-0.6);
    \node (labelm1) at (0.3,-0.8) {$0.3$};
    
    \draw[-,thick] (1,-0.4) -- (1,-0.6);
    \node (labelm1) at (1,-0.8) {$1$};

    \node[particle,text=white] (1) at (-1,0) {0};
    \node[particle,text=white] (2) at (0.3,0) {0};
    \node[particle,text=white] (3) at (0.3,0.75) {0};
    \node[particle,text=white] (4) at (1,0) {0};

\end{tikzpicture}

%% file: figures/fig2-2.tikz
\begin{tikzpicture}
  [
  middlearrow/.style 2 args={
        decoration={             
            markings, 
            mark=at position 0.5 with {\arrow[xshift=3.333pt]{triangle 45}, \node[#1] {#2};}
        },
        postaction={decorate}
    },
    particle/.style={circle, draw},
    minimum size=5pt
    ]

    \node[label=left:{$\Phi_x(\lambda) :$}] at (5,-0.5) {};
    \draw[->,thick] (5,-0.5) -- (10,-0.5);
    
    \draw[-,thick] (6,-0.4) -- (6,-0.6);
    \node (labelm1) at (6,-0.8) {$-1$};
    
    \draw[-,thick] (7.3,-0.4) -- (7.3,-0.6);
    \node (labelm1) at (7.3,-0.8) {$0.3$};
    
    \draw[-,thick] (8,-0.4) -- (8,-0.6);
    \node (labelm1) at (8,-0.8) {$1$};
    
    \draw[-,thick] (9,-0.4) -- (9,-0.6);
    \node (labelm1) at (9,-0.8) {$2$};
    
    \node[particle,text=white] (1b) at (6,0) {0};
    \node[particle,text=white] (2b) at (7.3,0) {0};
    \node[particle,text=white] (3b) at (7.3,0.75) {0};
    \node[particle,text=white] (4b) at (8,0) {0};
    \node[particle,text=white] (5b) at (9,0) {0};

\end{tikzpicture}

%% file: figures/fig3-1.tikz
\begin{tikzpicture}[
    middlearrow/.style 2 args={
        decoration={             
            markings, 
            mark=at position 0.5 with {\arrow[xshift=3.333pt]{triangle 45}, \node[#1] {#2};}
        },
        postaction={decorate}
    },
    particle/.style={circle, draw},
    minimum size=5pt,
    scale=0.9
]


    \node[circle,fill=black,inner sep=0pt,minimum size=5pt,label=below:{$0$}] (0) at (6,0) {};
    \node[circle,fill=black,inner sep=0pt,minimum size=5pt,label=below:{$1$}] (1) at (8,0) {};
    \node[circle,fill=black,inner sep=0pt,minimum size=5pt,label=below:{$2$}] (2) at (10,0) {};
    \node[circle,fill=black,inner sep=0pt,minimum size=5pt,label=below:{$3$}] (3) at (12,0) {};
    \node[circle,fill=black,inner sep=0pt,minimum size=5pt,label=below:{$4$}] (4) at (14,0) {};

    \draw[black!100,middlearrow={below}{1}] (0) -- (1);
    \draw[black!100,middlearrow={below}{0}] (1) -- (2);
    \draw[black!100,middlearrow={below}{2}] (2) -- (3);
    \draw[black!100,middlearrow={below}{-3}] (3) -- (4);

    \draw[black!100] (0) edge[bend left, middlearrow={below}{0.7}] (2);
    \draw[black!100] (1) edge[bend left, middlearrow={below}{0.5}] (3);
    \draw[black!100] (2) edge[bend left, middlearrow={below}{0.7}] (4);
    
    \draw[black!100] (0) edge[bend right, middlearrow={below}{0.9}] (3);
    \draw[black!100] (1) edge[bend right, middlearrow={below}{0.2}] (4);

    \draw[black!100] (0) edge[bend left, middlearrow={below}{1}] (4);


\end{tikzpicture}

%% file: figures/fig3-2.tikz
\begin{tikzpicture}[
    middlearrow/.style 2 args={
        decoration={             
            markings, 
            mark=at position 0.5 with {\arrow[xshift=3.333pt]{triangle 45}, \node[#1] {#2};}
        },
        postaction={decorate}
    },
    particle/.style={circle, draw},
    minimum size=5pt
]


    \draw[->,thick] (16,-0.5) -- (22.5,-0.5);
    
    \draw[-,thick] (17,-0.4) -- (17,-0.6);
    \node (labelm1) at (17,-0.8) {$0$};
    
    \draw[-,thick] (18.5,-0.4) -- (18.5,-0.6);
    \node (labelm1) at (18.5,-0.8) {$1$};
    
    \draw[-,thick] (19.625,-0.4) -- (19.625,-0.6);
    \node (labelm1) at (19.625,-0.8) {$1.7$};

        \draw[-,thick] (21.5,-0.4) -- (21.5,-0.6);
    \node (labelm1) at (21.5,-0.8) {$3$};
    
    
    \node[particle] (1b) at (17,0) {0};
    \node[particle] (2b) at (18.5,0) {1};
    \node[particle] (3b) at (18.5,0.75) {2};
    \node[particle] (4b) at (21.5,0) {3};
    \node[particle] (5b) at (19.625,0) {4};

\end{tikzpicture}

%% file: figures/fig4-1.tikz
\begin{tikzpicture}[
    middlearrow/.style 2 args={
        decoration={             
            markings, 
            mark=at position 0.5 with {\arrow[xshift=3.333pt]{triangle 45}, \node[#1] {#2};}
        },
        postaction={decorate}
    },
    particle/.style={circle, draw},
    minimum size=5pt,
    scale=0.95
]
    \node[label=left:{$\Phi_x(\lambda) :$}] at (8.5,-0.5) {};
    \draw[->,thick] (8.5,-0.5) -- (12.5,-0.5);
    
    \draw[-,thick] (9,-0.4) -- (9,-0.6);
    \node (labelm1) at (9,-0.8) {$0$};
    
    \draw[-,thick] (10.5,-0.4) -- (10.5,-0.6);
    \node (labelm2) at (10.5,-0.8) {$1$};
    
    \draw[-,thick] (12,-0.4) -- (12,-0.6);
    \node (labelm3) at (12,-0.8) {$2$};
    
    
    \node[particle,text=white] (1) at (9,0) {0};
    \node[particle,blue] (2) at (10.5,0) {1};
    \node[particle,blue] (3) at (11.25,0) {2};
    \node[particle,blue] (4) at (12,0) {3};
    \node[particle,blue] (5) at (12,0.75) {4};

\end{tikzpicture}

%% file: figures/fig4-2.tikz
\begin{tikzpicture}[
    middlearrow/.style 2 args={
        decoration={             
            markings, 
            mark=at position 0.5 with {\arrow[xshift=3.333pt]{triangle 45}, \node[#1] {#2};}
        },
        postaction={decorate}
    },
    particle/.style={circle, draw},
    minimum size=5pt,
    scale=0.95
]
    \node[label=left:{$\Phi_x(\lambda') :$}] at (15,-0.5) {};
    \draw[->,thick] (15,-0.5) -- (20.5,-0.5);
    
    \draw[-,thick] (15.5,-0.4) -- (15.5,-0.6);
    \node (labelm1b) at (15.5,-0.8) {$2$};
    
    \draw[-,thick] (17,-0.4) -- (17,-0.6);
    \node (labelm1b) at (17,-0.8) {$3$};
    
    \draw[-,thick] (18.5,-0.4) -- (18.5,-0.6);
    \node (labelm2b) at (18.5,-0.8) {$4$};
    
    \draw[-,thick] (20,-0.4) -- (20,-0.6);
    \node (labelm3b) at (20,-0.8) {$5$};
    
    
    \node[particle,text=white] (1bbb) at (15.5,0) {0};
    \node[particle,text=white] (1bb) at (17,0.75) {0};
    \node[particle,text=white] (1b) at (17,0) {0};
    \node[particle, blue] (2b) at (18.5,0) {1};
    \node[particle, blue] (3b) at (19.25,0) {2};
    \node[particle, blue] (4b) at (20,0) {3};
    \node[particle, blue] (5b) at (20,0.75) {4};

\end{tikzpicture}

%% file: figures/fig9.tikz
\begin{tikzpicture}[
    middlearrow/.style 2 args={
        decoration={             
            markings, 
            mark=at position 0.5 with {\arrow[xshift=3.333pt]{triangle 45}, \node[#1] {#2};}
        },
        postaction={decorate}
    },
    scale=1.5
]

    \node[circle,fill=black,inner sep=0pt,minimum size=5pt,label=below:{$0$}] (0) at (6,0) {};
    \node[circle,fill=black,inner sep=0pt,minimum size=5pt,label=below:{$1$}] (1) at (8,0) {};
    \node[circle,fill=black,inner sep=0pt,minimum size=5pt,label=below:{$2$}] (2) at (10,0) {};
    \node[circle,fill=black,inner sep=0pt,minimum size=5pt,label=below:{$3$}] (3) at (12,0) {};

    \draw[black!100,middlearrow={below}{1}] (0) -- (1);
    \draw[black!100,middlearrow={below}{0}] (1) -- (2);
    \draw[black!100,middlearrow={below}{-0.2}] (2) -- (3);

    \draw[black!100] (0) edge[bend left, middlearrow={below}{1}] (2);
    \draw[black!100] (1) edge[bend left, middlearrow={below}{1}] (3);
    
    \draw[black!100] (0) edge[bend left, middlearrow={below}{0.9}] (3);

\end{tikzpicture}

%% file: figures/fig10.tikz
\begin{tikzpicture}
    \draw[->] (-6,0)--(2,0) node[anchor=north east] {$n$}; 
    \draw[->] (0,-1)--(0,5) node[anchor=north east] {$i$}; 
    \draw [dotted, gray] (-6,-1) grid (2,5);

    \foreach \x in {-6,-5,-4,-3,-2,-1,1}
        \draw (\x,3pt) -- (\x,-3pt)
            node[anchor=north] {$\x$};
    \draw (0,3pt) -- (0,-3pt)
            node[anchor=north west] {$0$};
    \foreach \y in {1,2,3,4}
        \draw (3pt,\y) -- (-3pt,\y)
            node[anchor=east] {$\y$};

    \foreach \Point in {(1,3), (0,1), (-1,2),(-2,1),(-2,3),(-3,3),(-1,4),(0,2),(-4,2),(-5,1)}{
    \node at \Point {\large\textbullet};
}

    \draw[red, dashed] (1,1)--(1,0);
    
    \draw[red, dashed] (0.05,1)--(1,2);
    \draw[red, dashed] (0.05,1)--(0.05,0);
    
    \draw[red, dashed] (-1,1)--(1,3);
    \draw[red, dashed] (-1,1)--(-1,0);
    
    \draw[blue, dashed] (-2,1)--(1,4);
    \draw[blue, dashed] (-2,1)--(-2,0);

    \draw[red, dashed] (-3,1)--(1,5);
    \draw[red, dashed] (-3,1)--(-3,0);

    \draw[red, dashed] (-4,1)--(0,5);
    \draw[red, dashed] (-4,1)--(-4,0);

    \draw[blue, dashed] (-5,1)--(-1,5);
    \draw[blue, dashed] (-5,1)--(-5,0);
    
\end{tikzpicture}

%% file: figures/fig8-1.tikz
\begin{tikzpicture}
  [
  middlearrow/.style 2 args={
        decoration={             
            markings, 
            mark=at position 0.5 with {\arrow[xshift=3.333pt]{triangle 45}, \node[#1] {#2};}
        },
        postaction={decorate}
    },
    particle/.style={circle, draw},
    minimum size=5pt
    ]
    
    \draw[->,thick] (-0.5,-0.5) -- (9.5,-0.5);
    
    \draw[-,thick] (0,-0.4) -- (0,-0.6);
    \node (labelm1) at (0,-0.9) {$0$};
    
    \draw[-,thick] (1,-0.4) -- (1,-0.6);
    
    \draw[-,thick] (2,-0.4) -- (2,-0.6);
    \node (labelm1) at (2,-0.9) {$M$};
    
    \draw[-,thick] (3,-0.4) -- (3,-0.6);
    
    \draw[-,thick] (4,-0.4) -- (4,-0.6);
    \node (labelm1) at (4,-0.9) {$2M$};
    
    \draw[-,thick] (5,-0.4) -- (5,-0.6);

    \draw[-,thick] (6,-0.4) -- (6,-0.6);
    \node (labelm1) at (6,-0.9) {$3M$};
    \draw[-,thick] (7,-0.4) -- (7,-0.6);
    \draw[-,thick] (8,-0.4) -- (8,-0.6);
    \node (labelm1) at (8,-0.9) {$4M$};
    \draw[-,thick] (9,-0.4) -- (9,-0.6);
    
    \node[particle] (0) at (0,0) {0};
    \node[particle] (1) at (2,0) {1};
    \node[particle] (2) at (4,0) {2};
    \node[particle] (3) at (6,0) {3};
    \node[particle] (4) at (8,0) {4};
    \node[particle] (5) at (2,0.75) {5};
    \node[particle] (6) at (4,0.75) {6};
    \node[particle] (7) at (6,0.75) {7};
    \node[particle] (8) at (8,0.75) {8};
    
    \draw[dotted, red] (4) edge[bend right, middlearrow={above}{$\leq -3M$}] (5);
    \draw[dotted, blue] (0) edge[middlearrow={below}{$+M$}] (1);
    \draw[dotted, blue] (1) edge[middlearrow={below}{$+M$}] (2);
    \draw[dotted, blue] (2) edge[middlearrow={below}{$+M$}] (3);
    \draw[dotted, blue] (3) edge[middlearrow={below}{$+M$}] (4);
    \draw[dotted, blue] (5) edge[middlearrow={below}{$+M$}] (6);
    \draw[dotted, blue] (6) edge[middlearrow={below}{$+M$}] (7);
    \draw[dotted, blue] (7) edge[middlearrow={below}{$+M$}] (8);
    \draw[dotted, blue] (0) edge[middlearrow={above}{$+M$}] (5);


\end{tikzpicture}

%% file: figures/fig8-2.tikz
\begin{tikzpicture}
  [
  middlearrow/.style 2 args={
        decoration={             
            markings, 
            mark=at position 0.5 with {\arrow[xshift=3.333pt]{triangle 45}, \node[#1] {#2};}
        },
        postaction={decorate}
    },
    particle/.style={circle, draw},
    minimum size=5pt
    ]
    
    \draw[->,thick] (-0.5,-0.5) -- (9.5,-0.5);
    
    \draw[-,thick] (0,-0.4) -- (0,-0.6);
    \node (labelm1) at (0,-0.9) {$0$};
    
    \draw[-,thick] (1,-0.4) -- (1,-0.6);
    
    \draw[-,thick] (2,-0.4) -- (2,-0.6);
    \node (labelm1) at (2,-0.9) {$M$};
    
    \draw[-,thick] (3,-0.4) -- (3,-0.6);
    
    \draw[-,thick] (4,-0.4) -- (4,-0.6);
    \node (labelm1) at (4,-0.9) {$2M$};
    
    \draw[-,thick] (5,-0.4) -- (5,-0.6);

    \draw[-,thick] (6,-0.4) -- (6,-0.6);
    \node (labelm1) at (6,-0.9) {$3M$};
    \draw[-,thick] (7,-0.4) -- (7,-0.6);
    \draw[-,thick] (8,-0.4) -- (8,-0.6);
    \node (labelm1) at (8,-0.9) {$4M$};
    \draw[-,thick] (9,-0.4) -- (9,-0.6);
    
    \node[particle] (0) at (0,0) {0};
    \node[particle] (1) at (2,0) {1};
    \node[particle] (2) at (4,0) {2};
    \node[particle] (3) at (6,0) {3};
    \node[particle] (4) at (8,0) {4};
    \node[particle] (5) at (3,0) {5};
    \node[particle] (6) at (5,0) {6};
    \node[particle] (7) at (7,0) {7};
    \node[particle] (8) at (9,0) {8};
    
    \draw[dotted,blue] (4) edge[bend right=50, middlearrow={above}{$-2.5M$}] (5);
    \draw[dotted,blue] (0) edge[bend left=50, middlearrow={above}{$+M$}] (1);
    \draw[dotted,blue] (1) edge[bend left=50, middlearrow={above}{$+M$}] (2);
    \draw[dotted,blue] (2) edge[bend left=50, middlearrow={above}{$+M$}] (3);
    \draw[dotted,blue] (3) edge[bend left=50, middlearrow={above}{$+M$}] (4);
    \draw[dotted,blue] (5) edge[bend left=50, middlearrow={above}{$+M$}] (6);
    \draw[dotted,blue] (6) edge[bend left=50, middlearrow={above}{$+M$}] (7);
    \draw[dotted,blue] (7) edge[bend left=50, middlearrow={above}{$+M$}] (8);


\end{tikzpicture}

%% file: figures/fig5-1.tikz
\begin{tikzpicture}
  [
  middlearrow/.style 2 args={
        decoration={             
            markings, 
            mark=at position 0.5 with {\arrow[xshift=3.333pt]{triangle 45}, \node[#1] {#2};}
        },
        postaction={decorate}
    },
    particle/.style={circle, draw},
    minimum size=5pt
    ]
    
    \node[label=left:{$\lambda^{(5)} :$}] at (-0.5,-0.5) {};
    \draw[->,thick] (-0.5,-0.5) -- (5.5,-0.5);
    
    \draw[-,thick] (0,-0.4) -- (0,-0.6);
    \node (labelm1) at (0,-0.9) {$0$};
    
    \draw[-,thick] (1,-0.4) -- (1,-0.6);
    \node (labelm1) at (1,-0.9) {$\frac{1}{4}$};
    
    \draw[-,thick] (2,-0.4) -- (2,-0.6);
    \node (labelm1) at (2,-0.9) {$\frac{1}{2}$};
    
    \draw[-,thick] (3,-0.4) -- (3,-0.6);
    \node (labelm1) at (3,-0.9) {$\frac{3}{4}$};
    
    \draw[-,thick] (4,-0.4) -- (4,-0.6);
    \node (labelm1) at (4,-0.9) {$1$};
    
    \draw[-,thick] (5,-0.4) -- (5,-0.6);
    \node (labelm1) at (5,-0.9) {$\frac{5}{4}$};

    \node[particle,text=white] (1) at (0,0) {0};
    \node[particle,text=white] (2) at (2,0) {0};
    \node[particle,text=white] (3) at (3,0) {0};
    \node[particle,text=white] (4) at (5,0) {0};



\end{tikzpicture}

%% file: figures/fig5-2.tikz
\begin{tikzpicture}
  [
  middlearrow/.style 2 args={
        decoration={             
            markings, 
            mark=at position 0.5 with {\arrow[xshift=3.333pt]{triangle 45}, \node[#1] {#2};}
        },
        postaction={decorate}
    },
    particle/.style={circle, draw},
    minimum size=5pt
    ]

\node[label=left:{$y^{(5)} :$}] at (10.5,-0.5) {};
    \draw[->,thick] (10.5,-0.5) -- (14.5,-0.5);
    
    \draw[-,thick] (11,-0.4) -- (11,-0.6);
    \node (labelm1) at (11,-0.9) {$-\frac{3}{4}$};
    
    \draw[-,thick] (12,-0.4) -- (12,-0.6);
    \node (labelm1) at (12,-0.9) {$-\frac{1}{2}$};
    
    \draw[-,thick] (13,-0.4) -- (13,-0.6);
    \node (labelm1) at (13,-0.9) {$-\frac{1}{4}$};
    
    \draw[-,thick] (14,-0.4) -- (14,-0.6);
    \node (labelm1) at (14,-0.9) {$0$};
    
    \node[particle,text=white] (4) at (14,0) {0};
    \node[particle,text=white] (4) at (12,0) {0};
    \node[particle,text=white] (4) at (11,0) {0};

\end{tikzpicture}

%% file: figures/fig6.tikz
\begin{tikzpicture}[
    middlearrow/.style 2 args={
        decoration={             
            markings, 
            mark=at position 0.5 with {\arrow[xshift=3.333pt]{triangle 45}, \node[#1] {#2};}
        },
        postaction={decorate}
    },
    scale=1.5
]

    \node[circle,fill=black,inner sep=0pt,minimum size=5pt,label=below:{$1$}] (1) at (8,0) {};
    \node[circle,fill=black,inner sep=0pt,minimum size=5pt,label=below:{$2$}] (2) at (10,0) {};
    \node[circle,fill=black,inner sep=0pt,minimum size=5pt,label=below:{$3$}] (3) at (12,0) {};
    \node[circle,fill=black,inner sep=0pt,minimum size=5pt,label=below:{$4$}] (4) at (14,0) {};
    \node[circle,fill=black,inner sep=0pt,minimum size=5pt,label=below:{$5$}] (5) at (16,0) {};
    \node[circle,fill=black,inner sep=0pt,minimum size=5pt,label=below:{$6$}] (6) at (18,0) {};    

    \draw[red,middlearrow={below}{$m$}] (1) -- (2);
    \draw[blue,middlearrow={below}{$1$}] (2) -- (3);
    \draw[blue,middlearrow={below}{$1$}] (3) -- (4);
    \draw[red,middlearrow={below}{$m$}] (4) -- (5);
    \draw[blue,middlearrow={below}{$1$}] (5) -- (6);

    \draw[blue] (1) edge[bend left, middlearrow={below}{$1$}] (3);
    \draw[blue] (2) edge[bend left, middlearrow={below}{$1$}] (4);
    \draw[blue] (3) edge[bend left, middlearrow={below}{$1$}] (5);
    \draw[red] (4) edge[bend left, middlearrow={below}{$m$}] (6);
    
    \draw[blue] (1) edge[bend right, middlearrow={below}{$1$}] (4);
    \draw[blue] (2) edge[bend right, middlearrow={below}{$1$}] (5);
    \draw[red] (3) edge[bend right, middlearrow={below}{$m$}] (6);

    \draw[red] (1) edge[bend left, middlearrow={above}{$m$}] (5);
    \draw[red] (2) edge[bend left, middlearrow={above}{$m$}] (6);
    
    \draw[red] (1) edge[bend right, middlearrow={below}{$m$}] (6);

\end{tikzpicture}

%% file: figures/fig7-1.tikz
\begin{tikzpicture}[
    middlearrow/.style 2 args={
        decoration={             
            markings, 
            mark=at position 0.5 with {\arrow[xshift=3.333pt]{triangle 45}, \node[#1] {#2};}
        },
        postaction={decorate}
    },
    particle/.style={circle, draw},
    minimum size=5pt,
    scale=0.95
]
    \node[label=left:{$\lambda^{(4)} :$}] at (8.5,-0.5) {};
    \draw[->,thick] (8.5,-0.5) -- (11.5,-0.5);
    
    \draw[-,thick] (9,-0.4) -- (9,-0.6);
    \node (labelm1) at (9,-0.8) {$0$};
    
    \draw[-,thick] (10,-0.4) -- (10,-0.6);
    \node (labelm2) at (10,-0.8) {$1$};
    
    \draw[-,thick] (11,-0.4) -- (11,-0.6);
    \node (labelm3) at (11,-0.8) {$2$};

    
    \node[particle,text=white] (1) at (9,0) {0};
    \node[particle,text=white] (2) at (9,0.75) {0};
    \node[particle,text=white] (3) at (10,0) {0};
    \node[particle,text=white] (4) at (10,0.75) {0};
    \node[particle,text=white] (5) at (10,1.5) {0};

\end{tikzpicture}

%% file: figures/fig7-2.tikz
\begin{tikzpicture}[
    middlearrow/.style 2 args={
        decoration={             
            markings, 
            mark=at position 0.5 with {\arrow[xshift=3.333pt]{triangle 45}, \node[#1] {#2};}
        },
        postaction={decorate}
    },
    particle/.style={circle, draw},
    minimum size=5pt,
    scale=0.95
]
    \node[label=left:{$\lambda^{(5)} :$}] at (8.5,-0.5) {};
    \draw[->,thick] (8.5,-0.5) -- (11.5,-0.5);
    
    \draw[-,thick] (9,-0.4) -- (9,-0.6);
    \node (labelm1) at (9,-0.8) {$0$};
    
    \draw[-,thick] (10,-0.4) -- (10,-0.6);
    \node (labelm2) at (10,-0.8) {$1$};
    
    \draw[-,thick] (11,-0.4) -- (11,-0.6);
    \node (labelm3) at (11,-0.8) {$2$};
    
    
    \node[particle,text=white] (1) at (9,0) {0};
    \node[particle,text=white] (2) at (9,0.75) {0};
    \node[particle,text=white] (3) at (10,0) {0};
    \node[particle,text=white] (4) at (10,0.75) {0};
    \node[particle,text=white] (5) at (10,1.5) {0};
    \node[particle,text=white] (6) at (11,0) {0};

\end{tikzpicture}

%% file: figures/fig7-3.tikz
\begin{tikzpicture}[
    middlearrow/.style 2 args={
        decoration={             
            markings, 
            mark=at position 0.5 with {\arrow[xshift=3.333pt]{triangle 45}, \node[#1] {#2};}
        },
        postaction={decorate}
    },
    particle/.style={circle, draw},
    minimum size=5pt,
    scale=0.95
]
    \node[label=left:{$\lambda^{(6)} :$}] at (8.5,-0.5) {};
    \draw[->,thick] (8.5,-0.5) -- (11.5,-0.5);
    
    \draw[-,thick] (9,-0.4) -- (9,-0.6);
    \node (labelm1) at (9,-0.8) {$0$};
    
    \draw[-,thick] (10,-0.4) -- (10,-0.6);
    \node (labelm2) at (10,-0.8) {$1$};
    
    \draw[-,thick] (11,-0.4) -- (11,-0.6);
    \node (labelm3) at (11,-0.8) {$2$};

    
    \node[particle,text=white] (1) at (9,0) {0};
    \node[particle,text=white] (2) at (9,0.75) {0};
    \node[particle,text=white] (3) at (10,0) {0};
    \node[particle,text=white] (4) at (10,0.75) {0};
    \node[particle,text=white] (5) at (10,1.5) {0};
    \node[particle,text=white] (6) at (11,0) {0};
    \node[particle,text=white] (7) at (11,0.75) {0};
\end{tikzpicture}